\documentclass[a4paper, reqno]{amsart}

\usepackage{graphicx} 

\usepackage[T1]{fontenc}
\usepackage[utf8]{inputenc}
\usepackage{mathtools}

\usepackage{enumitem}
\setlist[enumerate,1]{label=\rm(\arabic*)}
\setlist[enumerate,2]{label=\rm(\alph*)}
\setlist[enumerate,3]{label=\rm(\roman*)}

\usepackage[PS]{diagrams}
\usepackage{xcolor}
\usepackage{pgf}

\usepackage[version=4]{mhchem}
\mhchemoptions{layout=stacked}

\newcommand{\iso}{\cong}

\newcommand{\eg}{\textit{e.g.}}

\newcommand{\viz}{\textit{viz.}}
\newcommand{\ignore}[1]{\relax}

\newcommand{\bZ}{\mathbb{Z}}

\newcommand{\nbd}{\nobreakdash}
\newcommand{\id}{\ensuremath{\mathrm{id}}}
\newcommand{\inc}{\ensuremath{\mathrm{inc}}}
\newcommand{\tensor}{\otimes}

\newcommand{\FP}{{\upshape{(FP)}}}

\newcommand{\cT}[1]{\ce{^{$\R$}\mathbf{Nil}^{#1}(\Rd{0})}}

\newcommand{\nix}{\, \hbox{-} \,}

\numberwithin{equation}{section}

\newtheorem{theorem}[equation]{Theorem}
\newtheorem{corollary}[equation]{Corollary}
\newtheorem{proposition}[equation]{Proposition}
\newtheorem{lemma}[equation]{Lemma}

\theoremstyle{definition}
\newtheorem{definition}[equation]{Definition}
\newtheorem{remark}[equation]{Remark}

\newtheorem{example}[equation]{Example}

\let\oldtocsubsection=\tocsubsection
\renewcommand{\tocsubsection}[2]{\hspace{3.5em}\(\cdot\)~\oldtocsubsection{#1}{#2}}

\usepackage[hypertexnames=false,colorlinks=false,pdfborderstyle={/S/U/W 1}]{hyperref}

\newcommand{\cC}{\ch\vb}

\newcommand{\cK}{\mathcal{K}}

\newcommand{\cS}{\mathcal{S}}
\newcommand{\cX}{\mathcal{X}}
\newcommand{\cY}{\mathcal{Y}}
\newcommand{\cZ}{\mathcal{Z}}
\newcommand{\inv}{^{-1}}
\DeclareMathOperator{\im}{im}

\newcommand{\sdot}{\cS_{\bullet}}

\newcommand{\R}{R}
\newcommand{\Rd}[1]{\R_{#1}}
\newcommand{\Rpn}[1]{\R_{\geq {#1}}}
\newcommand{\Rp}{\Rpn{0}}
\newcommand{\Rnn}[1]{\R_{\leq {#1}}}
\newcommand{\Rn}{\Rnn{0}}

\newcommand{\bP}{\mathbb{P}}
\newcommand{\pp}{\bP^{1}}

\newcommand{\ch}{\mathrm{Ch}^\flat\,}

\newcommand{\vb}{\mathrm{Vect}(\pp)}
\newcommand{\rbPr}[1]{\mathrm{\bf P}(#1)}   
\newcommand{\chpr}[1]{\ch \rbPr{#1}}
\newcommand{\chp}{\chpr{R_0}}

\newcommand{\Mod}[1]{\mathrm{Mod}\text{\rm -}#1}

\DeclareMathOperator{\torus}{\mathfrak{H}}

\DeclareMathOperator{\coker}{coker}
\DeclareMathOperator{\cyl}{cyl}
\DeclareMathOperator{\cone}{cone}
\DeclareMathOperator*{\dirlim}{colim}
\newcommand{\invlim}{\lim\limits_{\leftarrow}}

\newcommand{\tw}[3]{#1(#2,\,#3)}

\newcommand{\OO}[2]{\tw{\mathcal{O}} {#1} {#2}}

\newcommand{\kay}{\prescript{R}{}{K_{-1}}}
\newcommand{\NK}[2]{\ce{^{$R$}{\mathrm{NK}}^{$#1$}_{$#2$}}}
\newcommand{\nil}[2]{\ce{^{$R$}{\mathrm{Nil}}^{$#1$}_{$#2$}}}
\newcommand{\sker}[2]{\ce{^{$R$}_{ker}$K$_{$#2$}}}
\newcommand{\scoker}[2]{\ce{^{$R$}_{coker}$K$_{$#2$}}}
\newcommand{\sd}{\mathrm{sd}}
\newcommand{\qand}{\quad\text{and}\quad}
\newcommand{\qqand}{\qquad\text{and}\qquad}

\renewcommand{\rho}{\varrho}

\renewcommand{\theta}{\vartheta}

\begin{document}

\title[The fundamental theorem for strongly $\bZ$-graded rings]%
{The ``fundamental theorem'' for the\\algebraic $K$-theory %
  of strongly $\bZ$-graded rings}

\date{\today}

\author{Thomas H\"uttemann}

\address{Thomas H\"uttemann\\ Queen's University Belfast\\ School of
  Mathematics and Physics\\ Mathematical Sciences Research Centre\\ Belfast
  BT7~1NN\\ UK}

\email{t.huettemann@qub.ac.uk}

\urladdr{https://t-huettemann.github.io/}
\urladdr{https://pure.qub.ac.uk/en/persons/thomas-huettemann}


\begin{abstract}
  The ``fundamental theorem'' for algebraic $K$-theory expresses the
  $K$-groups of a \textsc{Laurent} polynomial ring $L[t,t\inv]$ as a
  direct sum of two copies of the $K$-groups of~$L$ (with a degree
  shift in one copy), and certain ``nil'' groups of~$L$. It is shown
  here that a modified version of this result generalises to strongly
  $\bZ$\nbd-graded rings; rather than the algebraic $K$-groups of~$L$,
  the splitting involves groups related to the shift actions on the
  category of $L$-modules coming from the graded structure. (These
  actions are trivial in the classical case). The nil groups are
  identified with the reduced $K$-theory of homotopy nilpotent twisted
  endomorphisms, and analogues of \textsc{Mayer}-\textsc{Vietoris} and
  localisation sequences are established.
\end{abstract}

\maketitle

\tableofcontents

\section*{Introduction}

\subsection*{The fundamental theorem for the algebraic $K$-theory of rings}

The ``fundamental theorem'', also know as the
\textsc{Bass}-\textsc{Heller}-\textsc{Swan} formula, expresses the
algebraic $K$-groups of a \textsc{Laurent} polynomial ring
$\Rd{0}[t,t\inv]$ as a direct sum
\begin{equation}
  K_{q}\big( \Rd{0}[t,t\inv] \big) \iso K_{q} \Rd{0} \oplus
  K_{q-1}(\Rd{0}) \oplus \NK+q(\Rd{0}) \oplus
  \NK-q(\Rd{0}) \ .\label{eq:classical}
\end{equation}

It will be shown that this result largely depends on the
structure of $\Rd{0}[t,t\inv]$ as a $\bZ$\nbd-graded ring, and that a
similar splitting can be established in much greater generality.

By way of analogy we think of any $\bZ$\nbd-graded ring
$R = \bigoplus_{k \in \bZ} \Rd{k}$ as a substitute for a
\textsc{Laurent} polynomial ring, and consider the subrings
$\Rn = \bigoplus_{k \leq 0} \Rd{0}$ and
$\Rp = \bigoplus_{k \geq 0} \Rd{k}$ as substitutes for the polynomial
rings $\Rd{0}[t\inv]$ and~$\Rd{0}[t]$. If the ring~$\R$ is
\textit{strongly graded} in the sense that
$\Rd{k} \Rd{\ell} = \Rd{k+\ell}$ for all integers $k$ and~$\ell$, the
analogy is appropriate, and many results known for (\textsc{Laurent})
polynomial rings can be proved to hold for the more general
setting, \eg, the characterisation of finite domination \textit{via}
\textsc{Novikov} homology \cite{MR3704930}, the splitting for the
algebraic $K$\nbd-theory of the projective line \cite{P1}, and the
connection between finite domination and non-commutative localisation
\cite{more_on_fd}.

This is not idle play: the class of strongly $\bZ$\nbd-graded rings is
much bigger than the class of \textsc{Laurent} polynomial rings. One
specific example of a strongly $\bZ$\nbd-graded rings is
$K[A,B,C,D]/(AB+CD=1)$ where $K$ is a field, $\deg(A) = \deg(C) = 1$
and $\deg(B) = \deg(D) = -1$. (The only units in this ring are the
non-zero elements of~$K$ in degree~$0$.) A natural infinite family of
examples is formed by the \textsc{Leavitt} path algebras associated
with row-finite directed graphs without sink satisfying a certain
condition~$Y$ \cite[Theorem~1.3]{nystedt2020strongly}; this includes
all \textsc{Leavitt} path algebras associated with finite directed
graphs without sink.

\medbreak

Back to the fundamental theorem, the graded structure of
$R = \bigoplus_{k \in \bZ} \Rd{k}$ induces ``shift functors''
$s_{k} \colon M \mapsto M \tensor_{\Rd{0}} \Rd{k}$ on categories of
$\Rd{0}$\nbd-modules; in case of a \textit{strongly} graded ring,
these functors preserve projectivity and satisfy the relation
$s_{k} \circ s_{\ell} \iso s_{k + \ell}$. The key point is to measure
how non-trivial the resulting $\bZ$\nbd-action on algebraic
$K$\nbd-groups is, which is done by considering the kernel~$A_{q}$ and
cokernel~$B_{q}$ of the \textit{shift difference map}
\begin{displaymath}
  \sd_{*} = \id - s_{-1} \colon K_{q} \Rd{0} \rTo K_{q} \Rd{0} \ .
\end{displaymath}
It turns out that the groups~$A_{q-1}$ and~$B_{q}$ play the role of
$K_{q-1}\Rd{0}$ and $K_{q}\Rd{0}$ in the classical
formulation~\eqref{eq:classical} of the fundamental theorem. More
precisely, $K_{q}\R$ will be shown to be an extension of~$A_{q-1}$ by
a direct sum of~$B_{q}$ with two appropriately defined nil terms
(Theorem~\ref{thm:fundamental}); the nil terms are identified as the
reduced algebraic $K$-theory of categories of homotopy nilpotent
endomorphisms (Theorem~\ref{thm:NK_is_nil}).



\subsection*{Relation with other work}

The ``classical'' fundamental theorem for the higher algebraic
$K$\nbd-theory of rings has been proved by \textsc{Quillen} and
\textsc{Grayson} \cite{MR0574096}.  It has been extended from the
$K$\nbd-theory of rings to the $K$\nbd-theory of schemes by
\textsc{Thomason} and \textsc{Trobaugh} \cite[Theorem~6.6]{MR1106918},
and to the algebraic $K$\nbd-theory of spaces \cite{MR1829311} by
\textsc{Klein}, \textsc{Vogell}, \textsc{Waldhausen},
\textsc{Williams} and the author. A version for skew \textsc{Laurent}
polynomial rings has been discussed by \textsc{Yao}
\cite{MR1325783}. More recently, the result has been established by
\textsc{L\"uck} and \textsc{Steimle} for skew \textsc{Laurent}
extensions of additive categories \cite{MR3441110}, and by
\textsc{Fontes} and \textsc{Ogle} \cite{fontes2018fundamental} in the
context of $\mathbb{S}$\nbd-algebras. Most of the recent accounts
follow the pattern laid out in~\cite{MR1829311}, which is in turn
loosely based on~\cite{MR0574096}.

The present paper goes beyond previous generalisations inasmuch as it
moves the focus away from the very special case of (skew)
\textsc{Laurent} polynomial rings to the essential information
contained in the \textit{graded} structure of the ring extension
$\Rd{0} \subset \R$.

\subsection*{Structure of the paper}

The paper is structured in a way that avoids forward references in
proofs. \S1 provides an overview of notation and main results. \S2
discusses induced chain complexes. \S3 introduces finite domination,
an important finiteness condition for chain complexes. In \S\S4--5 the
``projective line'' and its $K$\nbd-theory are reviewed. \S6 is
devoted to an analysis of the ``nil terms'' in the fundamental
theorem. \S7 contains the proof that the ``fundamental square'' of the
projective line (roughly speaking relating the $K$\nbd-groups
of~$\Rn$, $\R$ and~$\Rp$ with those of the projective line) is
homotopy cartesian, which leads to a proof of the
\textsc{Mayer}-\textsc{Vietoris} sequence in~\S8 and a proof of the
fundamental theorem in~\S9. In \S10 we establish a ``localisation
sequence'' for algebraic $K$\nbd-theory, and finish the paper by
identifying the nil groups as the reduced $K$\nbd-theory of categories
of homotopy nilpotent twisted endomorphisms in~\S11.


\subsection*{Acknowledgements}


The basic ideas for this paper were developed during a research visit
of the author to Beijing Institute of Technology. Their hospitality and
financial support is greatly appreciated.

\section{Notation and main results}

\subsection*{Notation and conventions}

The word ``ring'' will always refer to an associative unital ring,
homomorphisms of rings respect the unit, and ``modules'' are
understood to be unital and right, unless otherwise specified. Let
$\R = \bigoplus_{k \in \bZ} \Rd{k}$ be a $\bZ$\nbd-graded unital ring,
so that $\Rd{k} \R_{\ell} \subseteq \Rd{k+\ell}$ for all
$k,\ell \in \bZ$.  (Here $\Rd{k} \Rd{\ell}$ is the set of finite sums
of products $xy$ with $x \in \Rd{k}$ and $y \in \Rd{\ell}$.) The
component $\Rd{0}$ is a subring of~$\R$ with the same unit element
\cite[Proposition~1.4]{MR593823}. Two further subrings of note are
\begin{displaymath}
  \Rn = \bigoplus_{k \leq 0} \Rd{k} \qquad \text{and} \qquad \Rp =
  \bigoplus_{k \geq 0} \Rd{k} \ .
\end{displaymath}
There are ring inclusions
\begin{displaymath}
  \Rn \lTo^{i^{-}} \Rd{0} \rTo^{i^{+}} \Rp %
  \quad \text{and} \quad %
  \Rn \rTo_{j^{-}} \R \lTo_{j^{+}} \Rp \ ,
\end{displaymath}
and ring homomorphisms given by projection
\begin{displaymath}
  \Rn \rTo^{p^{-}} \Rd{0} \lTo^{p^{+}} \Rp \ ;
\end{displaymath}
they satisfy the relations
\begin{displaymath}
  j^{-} \circ i^{-} = j^{+} \circ i^{+} \ , \quad%
  p^{-} \circ i^{-} = \id_{\Rd{0}}
  \qand %
  p^{+} \circ i^{+} = \id_{\Rd{0}} \ .
\end{displaymath}
These various maps are used to define induction functors for modules,
\begin{displaymath}
  i^{-}_{*} \colon P \mapsto P \tensor_{\Rd{0}} \Rn
\end{displaymath}
and its relatives $i^{+}_{*}$, $j^{\mp}_{*}$ and $p^{\mp}_{*}$; the
resulting maps on algebraic $K$-groups are denoted by the same symbols
as the functors.

\subsection*{Strongly $\bZ$-graded rings}

The $\bZ$\nbd-graded ring~$\R$ is called \textit{strongly graded} if
$\Rd{k} \Rd{\ell} = \Rd{k+\ell}$ for all $k,\ell \in \bZ$, or
equivalently, if $\Rd{1} \Rd{-1} = \Rd{0} = \Rd{-1} \Rd{1}$. This
ensures that the ring multiplication yields $\Rd{0}$\nbd-bimodule
isomorphisms $\Rd{k} \tensor_{\Rd{0}} \Rd{-k} \iso \Rd{0}$ and, more
generally, $\Rd{k} \tensor_{\Rd{0}} \Rd{\ell} \iso \Rd{k + \ell}$;
consequently, each $\Rd{k}$ is an invertible $\Rd{0}$\nbd-bimodule,
and hence a finitely generated projective (left and right)
$\Rd{0}$\nbd-module \cite[Proposition~1.6]{MR3704930}. Similarly, one
verifies \cite[Lemma~1.9]{MR3704930} that
\begin{displaymath}
  \Rnn q = \bigoplus_{k \leq q} \Rd{k} %
  \quad \text{and} \quad %
  \Rpn {-p} = \bigoplus_{k \geq -p} \Rd{k}
\end{displaymath}
are finitely generated projective (left and right) modules over $\Rn$
and~$\Rp$, respectively, for all $p,q \in \bZ$.

\subsection*{The groups $\NK\pm q (\Rd{0})$}

Let $\R$ be a $\bZ$\nbd-graded ring.

\begin{definition}
  We define the \textit{nil groups of~$\Rd{0}$ relative to~$\R$} as
  \begin{gather*}
    \NK-q (\Rd{0}) = \coker \big( i^{-}_{*} \colon K_{q} (\Rd{0}) \rTo
    K_{q} (\Rn) \big) \\
    \intertext{and} %
    \NK+q (\Rd{0}) = \coker \big( i^{+}_{*} \colon K_{q} (\Rd{0}) \rTo
    K_{q} (\Rp) \big) \ .
  \end{gather*}
\end{definition}

Thus we have a split short exact sequence
\begin{displaymath}
  0 \rTo K_{q} (\Rd{0}) \pile{\rTo^{i^{+}_{*}} \\ \lDashto_{p^{+}_{*}}}
  K_{q} (\Rp) \rTo \NK+q (\Rd{0}) \rTo 0
\end{displaymath}
resulting in isomorphisms
\begin{equation}
  \label{eq:split-nil}
  K_{q} (\Rp) \iso K_{q} (\Rd{0}) \oplus \NK+q (\Rd{0}) %
  \quad \text{and} \quad %
  \NK+q \iso %
  \ker \big( p^{+}_{*} \colon K_{q} (\Rp) \rTo K_{q} (\Rd{0}) \big)%
  \ .
\end{equation}
Of course these remarks hold for $\NK-q (\Rd{0})$
and~$K_{q} (\Rn)$ \textit{mutatis mutandis}.

\subsection*{Shifts and shift differences}

The $\bZ$\nbd-graded over-ring $\R$ of the ring~$\Rd{0}$ defines
endofunctors of the category of $\Rd{0}$\nbd-modules, the
\textit{shift functors}, by
\begin{displaymath}
  s_{j} \colon P \mapsto P \tensor_{\Rd{0}} \Rd{j} \qquad (j \in \bZ)
  \ .
\end{displaymath}
If $\R$ is \textit{strongly} graded then $\Rd{j}$ is an invertible
$\Rd{0}$\nbd-bimodule whence $s_{j}$ defines an auto-equivalence of
the category of finitely generated projective $\Rd{0}$\nbd-modules,
with resulting isomorphisms on algebraic $K$\nbd-groups
\begin{displaymath}
  s_{j} \colon K_{q} \Rd{0} \rTo K_{q} \Rd{0} \ .
\end{displaymath}

\begin{definition}
  \label{def:shift_kernel}
  Suppose that $\R$ is a strongly $\bZ$\nbd-graded ring. The $q$th
  \textit{shift difference map of $\Rd{0}$ relative to~$\R$} is
  defined as
  \begin{equation}
    \label{eq:def_sd}
    \sd_{*} = \id - s_{-1} \colon K_{q} \Rd{0} \rTo K_{q} \Rd{0}
    \ . 
  \end{equation}
  The kernel of $\sd_{*} \colon K_{q} \Rd{0} \rTo K_{q} \Rd{0}$ is
  called the $q$th \textit{shift kernel of~$\Rd{0}$ relative to~$\R$},
  and is denoted by the symbol $\sker+q \Rd{0}$. The cokernel of this
  map is called the $q$th \textit{shift cokernel of~$\Rd{0}$ relative
    to~$\R$}, and is denoted by the symbol $\scoker+q \Rd{0}$.
\end{definition}

It might be worth pointing out that we could just as well use the
shift functor $s_{1}$ in place of~$s_{-1}$; in view of the
$\Rd{0}$\nbd-bimodule isomorphism
$\Rd{-1} \tensor_{\Rd{0}} \Rd{1} \iso \Rd{0}$, this would lead to the
same description of the shift kernel, and to an isomorphic description
of the shift cokernel. In any case, the exact sequence
\begin{displaymath}
  0 \rTo \sker+q \Rd{0} \rTo K_{q} \Rd{0} \rTo[l>=2.5em]^{\mathrm{sd}_{*}}
  K_{q} \Rd{0} \rTo \scoker+q \Rd{0} \rTo 0
\end{displaymath}
determines the $q$th shift kernel and $q$th shift cokernel up to
canonical isomorphism.

\begin{remark}
  If there is an $\Rd{0}$-bimodule isomorphism $\Rd{-1} \iso \Rd{0}$
  then $\sd_{*}$ is the zero map whence
  $\sker+q \Rd{0} = \scoker+q \Rd{0} = K_{q} \Rd{0}$. This happens,
  for example, in case $\R$ is a \textsc{Laurent} polynomial ring
  $R = \Rd{0}[t,t\inv]$ with a central indeterminate~$t$.
\end{remark}

\subsection*{The first negative $K$-group}

Let $\R = \bigoplus_{k \in \bZ} \Rd{k}$ be a $\bZ$\nbd-graded ring.

\begin{definition}
  We define the first negative $K$-group of~$\Rd{0}$ relative to~$\R$ as
  \begin{displaymath}
    \kay (\Rd{0}) = \coker (j^{-}_{*} + j^{+}_{*}) \ ,
  \end{displaymath}
  where the ring inclusions $j^{-}$ and~$j^{+}$ induce the map
  \begin{displaymath}
    j^{-}_{*} + j^{+}_{*} \colon K_{0} (\Rn) \oplus K_{0} (\Rp)
    \rTo K_{0} (R) \ .
  \end{displaymath}
\end{definition}

In the case of a \textsc{Laurent} polynomial ring
$\Rd{0} \subseteq \Rd{0}[t,t\inv]$ this recovers the \textsc{Bass}
$K$\nbd-group
$K_{-1}(\Rd{0}) = \prescript{\Rd{0}[t,t\inv]}{}{K_{-1}} (\Rd{0})$.

\subsection*{The fundamental theorem}

The fundamental theorem expresses $K_{q}\big( \Rd{0}[t,t\inv] \big)$
as a direct sum of groups $K_{q}(\Rd{0})$, $K_{q-1}(\Rd{0})$,
$\NK+q(\Rd{0})$ and $\NK-q(\Rd{0})$. In the more general context of
$\Rd{0} \subseteq \R$ with $\R$ strongly $\bZ$\nbd-graded, the
fundamental result reads as follows:

\begin{theorem}[The fundamental theorem for the algebraic $K$-theory of strongly $\bZ$-graded rings]
  \label{thm:fundamental}
  Let $\R$ be a strongly $\bZ$\nbd-graded ring. There
  are short exact sequences of \textsc{abel}ian groups
  \addtocounter{equation}{-1}
  \begin{subequations}
    \begin{gather}
      \begin{multlined}
        0 \rTo \NK-q (\Rd{0}) \oplus \scoker+q (\Rd{0}) \oplus %
        \NK+q (\Rd{0}) \qquad\qquad \\[0.8em] %
        \qquad \qquad \rTo K_{q} (R) \rTo %
        \sker+{q-1} (\Rd{0}) 
        \rTo 0 \qquad \text{(for \(q>0\))}
      \end{multlined}
      \\
      \intertext{and}
      \begin{multlined}
        0 \rTo \NK-0 (\Rd{0}) \oplus \scoker+0 (\Rd{0}) \oplus %
        \NK+0 (\Rd{0}) \qquad \\[0.8em] %
        \qquad \qquad \rTo K_{0} (R) \rTo %
        \kay \Rd{0} \rTo 0 \ .
      \end{multlined}
    \end{gather}
  \end{subequations}
\end{theorem}

\subsection*{The Mayer-Vietoris sequence}

Let $\R$ be a strongly $\bZ$\nbd-graded ring. By analogy with
algebraic geometry, one can consider a ``projective line'' which is
obtained by ``gluing $\mathrm{spec}\,\Rn$ and~$\mathrm{spec}\,\Rp$
along their intersection~$\mathrm{spec}\,\R$''; more precisely, one
can define the analogue of the category of quasi-coherent sheaves on
the projective line as the category of certain diagrams of
modules. The projective line~$\pp$ in this sense was introduced and
its $K$\nbd-theory computed by \textsc{Montgomery} and the author
\cite{P1}; the relevant parts of the theory will be surveyed in
\S\ref{sec:p1} below.

\begin{theorem}[{\textsc{Mayer}-\textsc{Vietoris} sequence}]
  \label{thm:MV-sequence}
  Let $\R$ be a strongly $\bZ$\nbd-graded ring. There is a long exact
  sequence of algebraic $K$-groups
  \begin{displaymath}
    \begin{array}[h]{cccccc}
      &&&& \llap{\(\ldots\ldots\)\quad} \rTo^{\gamma} & K_{q+1} (R) \\[0.7em] %
      \rTo^{\delta} & K_{q} (\pp) %
        & \rTo^{\beta} & K_{q} (\Rn) \oplus K_{q} (\Rp) %
                        & \rTo^{\gamma} & K_{q} (R) \\[0.7em] %
      \rTo^{\delta} & K_{q-1} (\pp) %
        & \rTo^{\beta} & K_{q-1} (\Rn) \oplus K_{q-1} (\Rp) %
                        & \rTo^{\gamma} & K_{q-1} (R) \\[0.7em] %
      &&& \vdots \\
      \rTo^{\delta} & K_{0} (\pp) %
        & \rTo^{\beta} & K_{0} (\Rn) \oplus K_{0} (\Rp) %
                        & \rTo^{\gamma} & K_{0} (R) \\[0.7em] %
      &&{\qquad} & \rTo & \kay (\Rd{0}) & \rTo  0 \ . %
    \end{array}
  \end{displaymath}
\end{theorem}

\subsection*{Homotopy nilpotent twisted endomorphisms and the localisation sequence}

Let $\R$ be a strongly $\bZ$\nbd-graded ring. We define $\cT{+}$ to be
the category of pairs $(Z,\zeta)$, with $Z$ an $\Rd{0}$\nbd-finitely
dominated bounded chain complexes of projective $\Rd{0}$\nbd-modules,
and
$\zeta \colon Z \tensor_{\Rd{0}} \Rd{1} \rTo Z \tensor_{\Rd{0}}
\Rd{0}$
a homotopy nilpotent twisted endomorphism, cf.~\S\ref{sec:nil}. Let
$\nil+q(\Rd{0})$ denote the $q$th algebraic $K$\nbd-group of~$\cT{+}$
(with respect to the quasi-isomorphisms as weak equivalences and the
levelwise split monomorphisms as cofibrations).

\begin{theorem}
  \label{thm:loc-and-nil}
  Let $\R$ be a strongly $\bZ$\nbd-graded ring. The forgetful functor
  $(Z,\zeta) \mapsto Z$, defined on~$\cT{+}$, induces a homomorphism
  $\nil+q (\Rd{0}) \rTo K_{q} \Rd{0}$ with
  kernel~$\NK-{q+1}(\Rd{0})$. The groups $\nil+q (\Rd{0})$ fit into a
  long exact sequences of algebraic $K$\nbd-groups
  \begin{multline*}
    \ldots \rTo K_{q+1} \R \rTo \nil+q (\Rd{0}) \rTo^{\phi} K_{q}
    \Rp
    \rTo K_{q} R \\[0.5em]
    \rTo \ldots \rTo \nil+0 (\Rd{0}) \rTo^{\phi} K_{0} \Rp \rTo
    K_{0} \R
  \end{multline*}
  with $\phi$ induced by the functor sending $(Z, \zeta)$ to the
  $\Rp$\nbd-module complex~$Z$ with $\Rp$ acting through~$\zeta$.
\end{theorem}

There is a symmetric version involving a category~$\cT{-}$, using
$\Rd{-1}$ in place of~$\Rd{1}$, and its $K$\nbd-groups
$\nil-q(\Rd{0})$. The complete version of the result is formulated and
proved in Theorems~\ref{thm:localisation-sequence}
and~\ref{thm:NK_is_nil} below.

\subsection*{Noetherian regular rings}

Suppose that $\R$ is a $\bZ$\nbd-graded ring (possibly not strongly
graded). Suppose further that $\R$ is right \textsc{noether}ian and
regular. For this situation, \textsc{van den Bergh} \cite{MR859452}
established an exact sequence
\begin{displaymath}
  \ldots \rTo K_{q+1} \R \rTo K_{q}^{\mathrm{gr}} \R \rTo^{\sigma}
  K_{q}^{\mathrm{gr}} \R \rTo K_{q} \R \rTo \ldots \ ,
\end{displaymath}
where $K_{q}^{\mathrm{gr}} \R$ denotes the ``graded $K$\nbd-theory
of~$\R$'', that is, the $K$\nbd-theory of the category of finitely
generated projective $\bZ$\nbd-graded $\R$\nbd-modules; the map
$\sigma$ is the difference of the maps induced by the identity and the
shift map $M \mapsto M(1)$. In case $\R$ is \textit{strongly} graded,
the categories of finitely generated projective $\Rd{0}$\nbd-modules
and of finitely generated projective graded $\R$\nbd-modules are
equivalent via the functor $P \mapsto P \tensor_{\Rd{0}} \R$
(\textsc{Dade} \cite[Theorem~2.8]{MR593823}). Thus we have the long
exact sequence
\begin{displaymath}
  \ldots \rTo K_{q+1} \R \rTo K_{q} \Rd{0} \rTo^{\sd_{*}} K_{q} \Rd{0}
  \rTo K_{q} \R \rTo \ldots
\end{displaymath}
which we can split at the shift difference map to obtain:

\begin{theorem}
  Suppose that $\R$ is a strongly $\bZ$\nbd-graded right
  \textsc{noether}ian regular ring. There are short exact sequences
  (for $q > 0$)
  \begin{displaymath}
    0 \rTo \scoker+q (\Rd{0}) \rTo K_{q} (R) \rTo \sker+{q-1} (\Rd{0})
    \rTo 0 \ . \tag*{\qedsymbol}
  \end{displaymath}
\end{theorem}

Comparison with Theorem~\ref{thm:fundamental} suggests that
$\NK\pm{q}(\Rd{0}) = 0$ in this situation. We will not pursue the
issue here.

\section{Induced modules and chain complexes}

\subsection*{The Grothendieck group of a ring}

Let $L$ be a rings. The group~$K_{0}(L)$ is, by definition, the
\textsc{Grothendieck} group of the category $\rbPr L$ of finitely
generated projective $L$\nbd-modules; we denote the element
corresponding to the module~$P$ by the symbol~$[P]$. The cokernel of
the map $K_{0}(\bZ) \rTo K_{0}(L)$ induced by the induction functor
$M \mapsto M \tensor_{\bZ} L$ is the reduced \textsc{Grothen\-dieck}
group $K_{0} (\bZ \downarrow L)$ of~$L$, more usually denoted by the
symbol $\tilde K_{0}(L)$. We write the element corresponding to the
module~$P$ by $\alpha \big([P]\big)$. The following equivalences are
well known:
\begin{subequations}
  \begin{gather}
    [P] = [Q] \quad \text{in } K_{0}(L) %
    \qquad \Leftrightarrow \qquad %
    \exists k \geq 0 \colon P \oplus L^{k} \iso Q \oplus L^{k} \\ %
    \intertext{and} %
    \alpha \big([P]\big) = \alpha \big([Q]\big) %
    \quad \text{in } K_{0}(\bZ \downarrow L) %
    \qquad \Leftrightarrow \qquad %
    \exists k,\ell \geq 0 \colon P \oplus L^{k} \iso Q \oplus
    L^{\ell} %
  \end{gather}
\end{subequations}
In particular, $\alpha \big([P]\big) = 0$ if and only if $P$~is stably
free.

\medbreak

We will repeatedly make use of the fact that $K_{0}(L)$ can be
described in other ways, for example using the machinery of
\textsc{Waldhausen} $K$\nbd-theory applied to the category $\chpr L$
of bounded complexes of finitely generated projective $L$\nbd-modules
(with quasi-isomorphisms as weak equivalences, and levelwise split
monomorphisms as cofibrations). This description is such that a chain
complex~$C$ in~$\chpr L$ gives rise to the element
$[C] = \chi (C) = \sum_{k} (-1)^{k} [C_{k}]$ of~$K_{0}(L)$. If $C$ is
contractible then $[C] = 0$.

\subsection*{Induced and stably induced modules}

Let $f \colon L \rTo S$ be a ring homomorphism, with induced map
\begin{displaymath}
f_{*} \colon K_{0} (L) \rTo K_{0} (S) \ , \quad [P] \mapsto [P
\tensor_{L} S] \ .
\end{displaymath}
We will use the notation $f_{*}$ also for the induction functor
$\nix \tensor_{L} S$ so that
$f_{*} \big( [P] \big) = \big[ f_{*}(P) \big]$.

\begin{definition}
  Let $Q$ be a finitely generated projective $S$\nbd-module.
  \begin{enumerate}
  \item We say that $Q$ is \textit{induced from~$L$} if there exists a
    finitely generated projective $L$\nbd-module $P$ such that
    $f_{*}(P) \iso Q$.
  \item The module $Q$ is called \textit{stably induced from~$L$} if
    there exist a number $q \geq 0$ and a finitely generated
    projective $L$\nbd-module $P$ such that
    $f_{*}(P) \iso Q \oplus S^{q}$.
  \end{enumerate}
\end{definition}

Algebraic $K$\nbd-theory provides an obstruction to modules being
stably induced, the obstruction group being
$K_{0} (L \downarrow S) = \coker (f_{*})$. It comes equipped with a
canonical map
\begin{equation}
  \label{eq:alpha}
  \alpha \colon K_{0} (S) \rTo K_{0} (L \downarrow S) \ .
\end{equation}

\begin{proposition}
  \label{prop:stably_induced_module}
  Let $Q$ be a finitely generated projective $S$-module. The following
  are equivalent:
  \begin{enumerate}
  \item The module $Q$ is stably induced from~$L$.
  \item The element $\alpha \big( [Q] \big)$
    of~$K_{0} (L \downarrow S)$ is trivial.
  \end{enumerate}
\end{proposition}

For example, if $Q$ is a stably free $S$\nbd-module then $Q$~is
certainly stably induced from~$L$ so that $\alpha \big( [Q] \big) = 0$
in~$K_{0} (L \downarrow S)$

\begin{proof}[Proof of Proposition~\ref{prop:stably_induced_module}]
  If $Q \oplus S^{s} = P \tensor_{L} S$ then
  $[Q] = f_{*} \big( [P] \big) - f_{*} \big( [L^{s}] \big)$ lies in
  the image of~$f_{*}$ whence
  $\alpha \big( [Q] \big) = 0 \in K_{0} (L \downarrow S)$.

  Conversely, suppose that $\alpha \big( [Q] \big) = 0$. Then there
  exists $a \in K_{0} (L)$ such that $[Q] = f_{*} (a)$ in~$K_{0} (S)$.
  We can find a finitely generated projective $L$-module $M$ and a
  number $r \geq 0$ such that $a = [M] - [L^{r}]$ in $K_{0} (L)$.
  By applying~$f_{*}$ and re-arranging we find the equality
  \[ [Q \oplus S^{r}] = [Q] + [S^{r}] = f_{*}(a) + f_{*} \big( [L^{r}]
  \big) = f_{*} \big( [M] \big) \quad \in K_{0} (S) \ .\]
  This in turn implies that there exists $k \geq 0$ with
  \begin{displaymath}
    Q \oplus S^{r} \oplus S^{k} \iso f_{*}(M) \oplus \ S^{k} = f_{*}
    (M \oplus L^{k}) \ .
  \end{displaymath}
  This shows $Q$ to be stably induced.
\end{proof}

\subsection*{Stabilisation of chain complexes}

As a matter of notation, we write $D(k,M)$ for the chain complex
concentrated in degrees $k$ and~$k-1$ with non-trivial entries~$M$ and
differential the identity map of~$M$.

\begin{definition}
  Let $D$ be a chain complex of $S$\nbd-modules. We say that the chain
  complex~$D'$ is a \textit{stabilisation} of~$D$ if
  $D' = D \oplus \bigoplus_{k} D(k, F_{k})$ for finitely many finitely
  generated free $S$\nbd-modules $F_{k}$.
\end{definition}

If $D'$ is a stabilisation of~$D$ then there are mutually homotopy
inverse chain homotopy equivalences $s \colon D \rTo D'$ (the
inclusion) and $r \colon D' \rTo D$ (the projection) such that
$rs = \id_{D}$, and such that
$\coker(s) = \ker(r) = \bigoplus_{k} D(k, F_{k})$ is a contractible
bounded complex with finitely generated free chain, boundary and cycle
modules.

\subsection*{Induced and stably induced chain complexes. The chain
  complex lifting problem}

Let $f \colon L \rTo S$ be a ring homomorphism.

\begin{definition}
  Let $D$ be a bounded complex of finitely generated projective
  $S$\nbd-modules.
  \begin{enumerate}
  \item We say that $D$ is \textit{induced from~$L$} if there exists a
    bounded complex of finitely generated projective $L$\nbd-module
    $C$ such that $f_{*}(C) \iso D$.
  \item The complex $D$ is called \textit{stably induced from~$L$} if
    there exist a bounded complex $C$ of finitely generated
    projective $L$\nbd-modules such that $f_{*}(C)$ is isomorphic to a
    stabilisation of~$D$.
  \end{enumerate}
\end{definition}

Again, let $D$ be a bounded complex of finitely generated projective
$S$\nbd-modules; suppose that all chain modules~$D_{k}$ are induced
from~$L$. The \textit{chain complex lifting problem} is to decide
whether $D$ is homotopy equivalent to a complex induced from~$L$. We
will address two variations of the theme below and show that
\begin{itemize}
\item if $D$ is acyclic, then there exists an \textit{acyclic} bounded
  complex~$C$ of finitely generated projective $L$\nbd-modules such
  that $f_{*}(C)$ is isomorphic to a stabilisation of~$D$;
\item if $f$ satisfies a certain strong flatness condition, then $D$
  is stably induced from~$L$.
\end{itemize}

\subsection*{The chain complex lifting problem for acyclic complexes}

\begin{theorem}
  \label{thm:lifting_acyclic_complexes}
  Every acyclic complex of stably induced modules is stably induced
  from an acyclic complex. --- More precisely, let $D$ be an acyclic
  bounded chain complex of finitely generated projective
  $S$\nbd-modules concentrated in chain levels $0$ to~$n$ such that
  each chain module $D_{k}$ is stably induced from~$L$. Then there
  exist a stabilisation~$D'$ of~$D$ and an acyclic bounded
  complex $C'$ of finitely generated projective $L$\nbd-modules, both
  concentrated in chain levels $0$ to~$n$, such that $f_{*} (C')$ is
  isomorphic to~$D'$.
  \end{theorem}

\begin{proof}
  As $D$ is acyclic there are finitely generated projective
  $S$-modules $E_{k}$, for $1 \leq k \leq n$, such that
  \begin{equation}
    \label{eq:this_is_D}
    D \iso \bigoplus_{k=1}^{n} D(k, E_{k}) \ .
  \end{equation}
  By construction we have $E_{1} = D_{0}$ and consequently
  $\alpha \big( [E_{1}] \big) = 0$, since $D_{0}$ is stably induced
  from~$L$ by hypothesis. By iteration, assuming that
  $\alpha \big( [E_{k-1}] \big) = 0$ is known, we infer from
  $D_{k-1} = E_{k-1} \oplus E_{k}$ that
  $\alpha \big( [E_{k}] \big) = 0$ as well.

  Thus for $1 \leq k \leq n$ we can choose numbers $j_{k} \geq 0$ and
  finitely generated projective $L$-modules $P_{k}$ such that
  \begin{displaymath}
    E_{k} \oplus S^{j_{k}} \iso f_{*}(P_{k}) \ .
  \end{displaymath}
  We define chain complexes
  \begin{displaymath}
    D' = \bigoplus_{k=1}^{n} D(k, E_{k} \oplus S^{j_{k}}) %
    \qquad \text{and} \qquad %
    C' = \bigoplus_{k=1}^{n} D(k, P_{k}) \ ; %
  \end{displaymath}
  thus $C'$ and~$D'$ are acyclic complexes concentrated in chain
  levels $0$ to~$n$. The computation
  \begin{displaymath}
    D' = \bigoplus_{k=1}^{n} D(k, E_{k} \oplus S^{j_{k}}) = %
    \bigoplus_{k=1}^{n} \Big( D(k, E_{k}) \oplus D(k,S^{j_{k}})
    \Big) %
    \underset{\eqref{eq:this_is_D}}{\iso} %
    D \oplus \bigoplus_{k=1}^{n} D(k, S^{j_{k}})
  \end{displaymath}
  confirms $D'$ as a stabilisation of~$D$. Finally, by construction
  we have isomorphisms of chain complexes
  \begin{displaymath}
    f_{*} (C') = \bigoplus_{k=1}^{n} D\big( k, f_{*}(P_{k}) \big) \iso
    \bigoplus_{k=1}^{n} D(k, E_{k} \oplus S^{j_{k}}) = D'
  \end{displaymath}
  as required.
\end{proof}

\subsection*{The chain complex lifting problem for well-behaved ring homomorphisms}

As before, let $f \colon L \rTo S$ be a ring homomorphism. We consider
$S$ as an $L$-$L$-bimodule, with $L$ acting \textit{via}~$f$. To solve
the chain complex lifting problem, we will assume that the following
condition is satisfied:
\begin{equation}
  \begin{gathered}
    \text{The $L$-$L$-bimodule $S$ is a filtered colimit of
      $L$-$L$-sub-bimodules~$S_{j}$ which} \\ %
    \text{are finitely generated projective right
      $L$\nbd-modules such that $S_{j} \tensor_{L} S = S\,$.}
  \end{gathered}
  \label{eq:condition}
\end{equation}
In particular, $S$ is a flat right $L$\nbd-module in this case.

\begin{example}
  \label{ex:strong_implies_diamond}
  The ring inclusion $f \colon L = \Rn \rTo^{\subseteq} \R = S$ with
  $\R$ a strongly $\bZ$-\nbd-graded ring satisfies
  condition~\eqref{eq:condition} since
  $\R = \bigcup_{k \geq 0} \Rnn {k}$ with $\Rnn {k}$ a finitely
  generated projective (left and right) $\Rn$\nbd-module
  \cite[Lemma~I.2.2]{P1} such that $\Rpn{k} \tensor_{\Rn} \R = \R$
  \cite[Lemma~I.2.4]{P1}.
\end{example}

\begin{proposition}
\label{prop:induced_complex}
Let $D$ be a chain complex consisting of finitely generated
pro\-jec\-tive $S$\nbd-mod\-ules concentrated in chain levels~$0$
to~$n$. Suppose that each chain module $D_{k}$ is induced from~$L$ so
that $D_{k} \iso f_{*}(C'_{k})$ for a finitely generated projective
$L$\nbd-module $C'_{k}$. Suppose further that the ring
homomorphism~$f$ satisfies condition~\eqref{eq:condition}. Then there
exists a chain complex $C$ of finitely generated projective
$L$\nbd-modules, concentrated in chain levels~$0$ to~$n$, such that
$f_{*}(C) \iso D$.
\end{proposition}

\begin{proof}
  We denote the differentials of~$D$ by
  $d_{k} \colon D_{k} \rTo D_{k-1}$. Set $C_{k} = \{0\}$ for $k<0$
  and for $k > n$, and choose $C_{n} = C'_{n}$.

  Since $C_{n}$ is finitely generated, the composite map
  \begin{displaymath}
    C_{n} \rTo C_{n} \tensor_{L} S \iso D_{n} \rTo^{d_{n}} D_{n-1} =
    f_{*} (C'_{n-1}) \underset{\eqref{eq:condition}}= \dirlim_{j}
    C'_{n-1} \tensor_{L} S_{j}   
  \end{displaymath}
  factors through some stage $j_{n}$ of the colimit system, resulting
  in a map
  \begin{displaymath}
    \partial_{n} \colon C_{n} \rTo C'_{n-1} \tensor_{L} S_{j_{n}} =:
    C_{n-1}
  \end{displaymath}
  with target a finitely generated projective $L$\nbd-module, by our
  hypotheses on~$f$, such that $f_{*} (\partial_{n}) \iso d_{n}$.

  Since $C_{n-1}$ is finitely generated, the composite map
  \begin{displaymath}
    C_{n-1} \rTo C_{n-1} \tensor_{L} S \iso D_{n-1} \rTo^{d_{n-1}}
    D_{n-2} =
    f_{*} (C'_{n-2}) \underset{\eqref{eq:condition}}= \dirlim_{j}
    C'_{n-2} \tensor_{L} S_{j}
  \end{displaymath}
  factors through some stage $j$ of the colimit system, resulting
  in a map
  \begin{displaymath}
    \tilde\partial_{n-1} \colon C_{n-1} \rTo C'_{n-2} \tensor_{L}
    S_{j} \ .
  \end{displaymath}
  By replacing $j$ with a larger index~$j_{n-1}$ (where ``larger''
  refers to the filtered properties of the indexing category), we may
  assume that the composite with $\partial_{n}$ is the zero map. In
  other words, we have a map
  \begin{displaymath}
    \tilde\partial_{n-1} \colon C_{n-1} \rTo C'_{n-2} \tensor_{L}
    S_{j_{n-1}} =: C_{n-2}
  \end{displaymath}
  with target a finitely generated projective $L$\nbd-module, by our
  hypotheses on~$f$, such that $f_{*} (\partial_{n}) \iso d_{n}$ and
  $\partial_{n-1} \circ \partial_{n} = 0$.

  The process is repeated iteratively, until we have constructed
  $C_{0}$ and~$\partial_{1}$.  
\end{proof}

\begin{proposition}
  \label{prop:make_induced}
  Let $D$ be a chain complex consisting of finitely generated
  projective $S$\nbd-mod\-ules concentrated in chain levels
  $0 \leq k \leq n$, such that each chain module $D_{k}$ is stably
  induced from~$L$.  Suppose that the ring homomorphism~$f$ satisfied
  condition~\eqref{eq:condition}. Then there exist a stabilisation
  $D'$ of~$D$ and a chain complex~$C'$ of finitely generated
  projective $L$\nbd-modules, both concentrated in chain levels
  $0 \leq k \leq n+1$, such that $f_{*}(C') \iso D'$.
\end{proposition}

\begin{proof}
  For $0 \leq k \leq n$ choose a number $s_{k} \geq 0$ such that
  $D_{k} \oplus S^{s_{k}}$ is induced from~$L$, and choose a finitely
  generated projective $L$\nbd-module $C_{k}$ with
  $f_{*} (C_{k}) \iso D_{k} \oplus S^{s_{k}}$. Set
  $D' = D \oplus \bigoplus_{k} D(k+1, S^{s_{k}})$. By construction,
  $D'$ is concentrated in chain levels $0$ to~$n+1$, and each chain
  module $D'$ is induced from~$L$ (as it is the direct sum of a module
  induced from~$L$ with a finitely generated free module). Hence
  Proposition~\ref{prop:induced_complex} yields a bounded complex~$C'$
  of finitely generated projective $L$\nbd-modules, concentrated in
  chain levels $0$ to~$n+1$, such that $f_{*}(C') \iso D'$.
\end{proof}

\section{Finite domination}

Let $C$ be a (possibly unbounded) chain complex of $K$\nbd-modules,
for some unital ring~$K$. We say that $C$ is \textit{$K$\nbd-finitely
  dominated}, or \textit{of type~\FP{} over~$K$}, if $K$ is homotopy
equivalent to a bounded complex of finitely generated projective
$K$\nbd-modules.

Finite domination satisfies the ``2-of-3 property'' with respect to
short exact sequences of chain complexes. We include a proof for
completeness.

\begin{lemma}
  \label{lem:K-fd-2-of-3}
  Let $C$, $D$ and~$E$ be bounded below complexes of projective
  $K$\nbd-modules, and suppose that there is a short exact sequence
  \begin{equation}
    \label{eq:K-fd-2-of-3}
    0 \rTo C \rTo^{f} D \rTo^{g} E \rTo 0 \ .
  \end{equation}
  If any two of the complexes are $K$\nbd-finitely dominated then so
  is the third.
\end{lemma}

\begin{proof}
  The given sequence gives rise to a short exact sequence of bounded
  complexes of projective $K$\nbd-modules
  \begin{displaymath}
    0 \rTo D \rTo \cyl(g) \rTo^{h} \cone(g) \rTo 0 \ ,
  \end{displaymath}
  together with homotopy equivalences $E \rTo \cyl(g)$ and
  $q \colon C[1] \rTo \cone(g)$. By iteration, there is a short exact
  sequence
  \begin{displaymath}
    0 \rTo \cyl(g) \rTo \cyl(h) \rTo \cone(h) \rTo 0 \ ,
  \end{displaymath}
  together with homotopy equivalences $C[1] \rTo \cyl(h)$ and
  $D[1] \rTo \cone(h)$. As finite domination is invariant under
  suspension and homotopy equivalences, it suffices to prove that
  \textit{if $C$ and~$D$ are $K$\nbd-finitely dominated so
    is~$E$}.

  Since we are dealing with bounded below complexes of projective
  modules, the canonical map $\cone(f) \rTo E$ is a homotopy
  equivalence.  As $C$ and~$D$ are $K$\nbd-finitely dominated there
  exist bounded complexes $C'$ and~$D'$ of finitely generated
  projective $K$\nbd-modules and chain homotopy equivalences
  $\alpha \colon C \rTo C'$ and $\beta \colon D \rTo D'$.  Choose a
  homotopy inverse~$\alpha'$ of~$\alpha$, and let
  $h \colon \id_{C} \simeq \alpha'\alpha$ be a homotopy. The chain map
  \begin{displaymath}
    \begin{pmatrix}
      \alpha \\
      \beta f h & \beta
    \end{pmatrix}
    \colon \cone(f) \rTo \cone(\beta f \alpha')
  \end{displaymath}
  is a quasi-isomorphism (since $\alpha$ and~$\beta$ are) and hence a
  homotopy equivalence; its target is a bounded complex of finitely
  generated projective $K$\nbd-modules. Thus
  $E \simeq \cone(f) \simeq \cone(\beta f \alpha')$ shows that $E$ is
  $K$\nbd-finitely dominated.
\end{proof}

\section{The projective line associated with a strongly $\bZ$-graded ring}

\label{sec:p1}

The projective line associated with a strongly $\bZ$-graded ring has
been introduced by \textsc{Montgomery} and the author \cite{P1}. We
recall definitions and $K$\nbd-theoretical results which will take a
central place when establishing the fundamental theorem.

\medbreak

\textbf{From now on we will assume throughout that
  $\R = \bigoplus_{k \in \bZ} R_{k}$ is a strongly $\bZ$\nbd-graded
  ring unless other hypotheses are specified.}

\begin{definition}[Sheaves and vector bundles on the projective line 
  {\cite[Definitions~II.1.1 and~II.2.1]{P1}}] 
  A {\it quasi-coherent sheaf on~$\pp$}, or just {\it sheaf} for
  short, is a diagram
  \begin{equation}
    \label{eq:sheaf}
    \cY = \quad \Big( Y^{-} \rTo^{\upsilon^{-}} Y^{0} \lTo^{\upsilon^{+}}
    Y^{+} \Big)
  \end{equation}
  where $Y^{-}$, $Y^{0}$ and $Y^{+}$ are modules over $\Rn$, $\R$
  and~$\Rp$, respectively, with an $\Rn$\nbd-linear homomorphisms
  $\upsilon^{-}$ and an $\Rp$\nbd-linear homomorphism $\upsilon^{+}$,
  such that the diagram of the adjoint $\R$\nbd-linear maps
  \begin{equation}
    \label{eq:sheaf_cond}
    Y^{-} \tensor_{\Rn} \R \rTo[l>=3em]^{\upsilon^{-}_\sharp}_{\iso} Y^{0}
    \lTo[l>=3em]^{\upsilon^{+}_{\sharp}}_{\iso} Y^{+} \tensor_{\Rp} \R
  \end{equation}
  consists of isomorphisms. This latter condition will be referred to
  as the \textit{sheaf condition}. A \textit{morphism}
  $f=(f^{-}, f^{0}, f^{+}) \colon \cY \rTo \cZ$ between sheaves is a
  commutative diagram of the form
  \begin{diagram}
    \cY &\quad & %
    Y^{-} \quad & \rTo^{\upsilon^{-}} & Y^{0} & \lTo^{\upsilon^{+}} &
    Y^{+} \\
    \dTo && \dTo<{f^{-}} && \dTo<{f^{0}} && \dTo<{f^{+}} \\
    \cZ && %
    Z^{-} & \rTo^{\zeta^{-}} & Z^{0} & \lTo^{\zeta^{+}} & Z^{+}
  \end{diagram}
  with $f^{-}$, $f^{0}$ and~$f^{+}$ homomorphisms of modules over
  $\Rn$, $\R$ and~$\Rp$, respectively.

  We call the sheaf $\cY$ a {\it vector bundle\/} if its constituent
  modules are finitely generated projective modules over their
  respective ground rings. The category of vector bundles (and all
  morphisms of sheaves between them) is denoted by~$\vb$.
\end{definition}

Specific examples of vector bundles are the {\it twisting sheaves\/}
\begin{displaymath}
  \OO k \ell = \quad \Big( \Rnn k \rTo^{\subseteq} \R \lTo^{\supseteq}
  \Rpn {-\ell} \Big)
\end{displaymath}
where $k$ and~$\ell$ are integers. (Note that the diagrams
$\OO k \ell$ may not be sheaves if the $\bZ$\nbd-graded ring~$\R$
fails to be \textit{strongly} graded.) --- Taking tensor product with
twisting sheaves defines an interesting operation on the category of
sheaves on~$\pp$:

\begin{definition}[Twisting {\cite[Definition~II.2.4]{P1}}]
  Let $\cY$ be a sheaf, and let $k, \ell \in \bZ$. We define the
  {\it $(k,\ell)${\rm th} twist of~$\cY$}, denoted $\tw \cY k \ell$,
  to be the sheaf
  \begin{displaymath}
    \tw \cY k \ell = \quad \Big( Y^{-} \tensor_{\Rn} \Rnn k \rTo Y^{0} 
    \tensor_{R} R \lTo Y^{+} \tensor_{\Rp} \Rpn {-\ell} \Big) \ ,
  \end{displaymath}
  with structure maps induced by those of~$\cY$ and the inclusion
  maps.
\end{definition}

\begin{definition}
  \label{def:psi}
  Let $k, \ell \in \bZ$. The {\it $(k,\ell)$th canonical sheaf
    functor}~$\Psi_{k,\ell}$ is defined by
  \begin{displaymath}
    \Psi_{k, \ell} \colon \rbPr {\Rd{0}} \rTo \vb \ , \quad P \mapsto P
    \tensor \OO k \ell
  \end{displaymath}
  where $\rbPr {\Rd{0}}$ is the category of finitely generated
  projective $\Rd{0}$\nbd-modules, and the symbol $P \tensor \OO k \ell$ denotes
  the sheaf
  \begin{displaymath}
    P \tensor_{R_{0}} \Rnn k \rTo P \tensor_{R_{0}} \R \lTo P
    \tensor_{R_{0}} \Rpn {-\ell} \ .
  \end{displaymath}
\end{definition}

\begin{definition}
  \label{def:cohomology}
  The \textit{sheaf cohomology modules} of the sheaf~$\cY$
  of~\eqref{eq:sheaf} are defined by the exact sequence
  \begin{displaymath}
    0 \rTo H^{0} \cY %
    \rTo Y^{-} \oplus Y^{+} %
    \rTo[l>=3em]^{\upsilon^{-} - \upsilon^{+}} Y^{0} %
    \rTo H^{1} \cY \rTo 0 \ ,
  \end{displaymath}
  that is, $H^{0} \cY = \ker (\upsilon^{-} - \upsilon^{+})$ and
  $H^{1} \cY = \coker (\upsilon^{-} - \upsilon^{+})$. We will also use
  the notation $\Gamma \cY$ for $H^{0} \cY$ and speak of
  \textit{global sections} of~$\cY$.
\end{definition}

The $\Rd{0}$\nbd-modules $H^{0} \cY$ and $H^{1} \cY$ depend
functorially on~$\cY$. Considering $\cY$ as a diagram of
$\Rd{0}$\nbd-modules we have isomorphisms
$H^{q} \cY = \invlim{}^{q} \cY$ for $q=0,1$.

\medbreak

One can explicitly compute the sheaf cohomology of twisting sheaves by
direct inspection \cite[Proposition~II.3.4]{P1}. Similarly, one can
show:

\begin{proposition}
  \label{prop:Gamma_Psi}
  For $P \in \rbPr{\Rd{0}}$ and $k,\ell \in \bZ$ there is an
  isomorphism
  \begin{displaymath}
    \Gamma \Psi_{k,\ell} P \iso \bigoplus_{j=-\ell}^{k} s_{j} P
  \end{displaymath}
  where $s_{j} P = P \tensor_{\Rd{0}} \Rd{j}$. The isomorphism is
  natural in~$P$ so that there are isomorphisms of functors
  \begin{align*}
    \Gamma \circ \Psi_{k,\ell} & \iso 0 \quad \text{if \(k+\ell <
                                 0\)}, \\
    \Gamma \circ \Psi_{0,0} & \iso \id \ , \\
    \Gamma \circ \Psi_{k, -k} & \iso s_{k} \ .
  \end{align*}
  Moreover, if $k + \ell \geq -1$ then
  $H^{1} \circ \Psi_{k,\ell} \cY = 0$. \qed
\end{proposition}

\section{The algebraic $K$-theory of the projective line}

We let $\ch \vb$ denote the category of bounded chain complexes of
vector bundles; similarly, we denote by $\ch \vb_{0}$ the category of
bounded chain complexes in the category~$\vb_{0}$. A map~$f$ of vector
bundles will be called an {\it $h$-equiv\-a\-lence}, or a {\it
  quasi-isomorphism\/}, if $f^{?}$ is a quasi-isomorphism of chain
complexes of modules for each decoration $? \in \{-,0,+\}$. As weak
equivalences are defined homologically, they satisfy the saturation
and extension axioms. All categories mentioned have a cylinder functor
given by the usual mapping cylinder construction which satisfies the
cylinder axiom.

\begin{definition}[{\cite[Definition~III.1.1]{P1}}]
  \label{def:K_P1}
  The $K$-theory space of the projective line is defined to be
  \begin{displaymath}
    K(\pp) = \Omega |h\mathcal{S}_{\bullet} \ch\vb | \ ,
  \end{displaymath}
  where ``$h$'' stands for the category of $h$\nbd-equivalences.
\end{definition}

It is technically more convenient to use a certain subcategory
of~$\vb$, on which the global sections functor~$\Gamma$ is exact:

\begin{definition}
  The category $\vb_{0}$ is the full subcategory of~$\vb$ consisting
  of those objects $\cY$ satisfying
  \begin{displaymath}
    H^{1} \cY(k,\ell) = 0 \quad \text{for all \(k,\ell \in \bZ\) with
      \(k+\ell \geq 0\).}
  \end{displaymath}
\end{definition}

\begin{lemma}[{\cite[Corollary~III.1.4]{P1}}]
  \label{lem:reduce_to_vb0}
  The inclusion $\vb_{0} \subseteq \vb$ induces a homotopy equivalence
  \begin{displaymath}
    h \mathcal{S}_{\bullet} \ch \vb_{0} \rTo^{\simeq} h
    \mathcal{S}_{\bullet} \ch \vb  \ ,
  \end{displaymath}
  and hence a homotopy equivalence
  $\Omega |h \mathcal{S}_{\bullet} \ch \vb_{0}| \rTo^{\simeq}
  K(\pp)$. \qed
\end{lemma}

For $k+ \ell \geq -1$ the functor
$\Psi_{k,\ell} \colon \chp \rTo \vb_{0}$ is an exact functor between
\textsc{Waldhausen} categories (with quasi-isomorphisms and
$h$\nbd-equivalences as weak equivalences, and cofibrations the
monomorphisms with cokernel an object of the category under
consideration).

\begin{theorem}[{\cite[Theorem~III.5.1]{P1}}]
  \label{thm:HM}
  Suppose that $R = \bigoplus_{k \in \bZ} R_{k}$ is a strongly
  $\bZ$\nbd-graded ring. There is a homotopy equivalence of
  $K$\nbd-theory spaces
  \begin{displaymath}
    K(R_{0}) \times
    K(R_{0}) \rTo K(\pp)
  \end{displaymath}
  induced by the functor
  \begin{align*}
    \Psi_{-1,0} + \Psi_{0,0} \colon \chp \times \chp & \rTo \ch \vb_{0} \\
    (C,D) & \rlap{\ \(\mapsto\)} \hphantom{\rTo} \,\Psi_{-1,0}(C)
            \oplus \Psi_{0,0} (D) \ . \tag*{\qedsymbol}
  \end{align*}
\end{theorem}

\section{The nil terms}

\label{sec:nil}

\newcommand{\twend}{\mathrm{Tw}^{+}\mathrm{End}(\Rd{0})}
\newcommand{\twendm}{\mathrm{Tw}^{-}\mathrm{End}(\Rd{0})}

\subsection*{The category of twisted endomorphisms}

Let $M$ be an $\Rd{0}$\nbd-module. A \textit{(positive) twisted
  endomorphism of~$M$} is an $\Rd{0}$\nbd-linear map
  \begin{displaymath}
\alpha \colon M \tensor_{\Rd{0}} \Rd{1} \rTo M \tensor_{\Rd{0}} \Rd{0} \ .
\end{displaymath}
The collection of such pairs $(M, \alpha)$, for various modules~$M$
and their twisted endomorphisms, forms a category $\twend$; a morphism
$f \colon (M, \alpha) \rTo (N, \beta)$ is an $\Rd{0}$\nbd-linear map
$f \colon M \rTo N$ with
$\beta \circ (f \tensor \id_{R_1}) = (f \tensor \id_{\Rd{0}}) \circ
\alpha$.
--- The category $\twendm$ of (negative) twisted endomorphisms of the
form $M \tensor_{\Rd{0}} \Rd{-1} \rTo M \tensor_{\Rd{0}} \Rd{0}$ is
defined analogously.

\medbreak

In case of a strongly $\bZ$\nbd-graded ring~$\R$ we have the equality
$\Rd{n+1} = \Rd{n}\Rd{1}$, or (what is the same) the isomorphism
$\Rd{n+1} \iso \Rd{n} \tensor_{\Rd{0}} \Rd{1}$. Given an object
$(M, \alpha)$ of~$\twend$, we can thus recursively define the
\textit{$n$\textrm{\upshape th} iteration of~$\alpha$} to be the map
\begin{displaymath}
  \alpha^{(n)} \colon M \tensor_{\Rd{0}} \Rd{n} \rTo %
  M \tensor_{\Rd{0}} \Rd{0} \qquad (n \geq 0)
\end{displaymath}
determined by
\begin{displaymath}
  \alpha^{(0)} = \id  \qqand \alpha^{(1)} = \alpha \ ,  
\end{displaymath}
with $\alpha^{(n+1)}$ being the composition
\begin{multline*}
  M \tensor_{\Rd{0}} \Rd{n+1} \iso %
  M \tensor_{\Rd{0}} \Rd{n} \tensor_{\Rd{0}} \Rd{1} %
  \rTo^{\alpha^{(n)} \tensor \id_{\Rd{1}}} %
  M \tensor_{\Rd{0}} \Rd{0} \tensor_{\Rd{0}} \Rd{1} \\ \iso %
  M \tensor_{\Rd{0}} \Rd{1} \rTo^{\alpha} M \tensor_{\Rd{0}} \Rd{0} \ .
\end{multline*}
We say that the twisted endomorphism $\alpha$ of~$M$ is
\textit{nilpotent} if $\alpha^{(n)} = 0$ for $n \gg 0$.

\medbreak

Every $\Rp$\nbd-module~$M$ determines an object $(M, \mu)$
of~$\twend$, with $M$ considered as an $\Rd{0}$\nbd-module by
restriction of scalars, and $\mu$ defined by
$m \tensor r \mapsto mr \tensor 1$ for $m \in M$ and $r \in
\Rd{1}$. There results a functor
\begin{displaymath}
  \Phi = \Phi^{+} \colon \Mod{\Rp} \rTo \twend \ ;
\end{displaymath}
an analogous construction provides us with a functor
\begin{displaymath}
  \hphantom{\Phi =} \Phi^{-} \colon \Mod{\Rn} \rTo \twendm \ .
\end{displaymath}

\begin{lemma}
  The functors $\Phi=\Phi^{+}$ and $\Phi^{-}$ are isomorphisms of
  categories.
\end{lemma}

\begin{proof}
  This is just the theory of modules over tensor rings, since (thanks
  to the strong grading) $\Rp$ is the tensor ring of~$\Rd{1}$
  over~$\Rd{0}$, and $\Rn$ is the tensor ring of~$\Rd{-1}$
  over~$\Rd{0}$.
\end{proof}

As a matter of notation, for $(M, \alpha) \in \twend$ we will denote
the $\Rp$\nbd-module $\Phi\inv(M, \alpha)$ simply by the symbol~$M$;
if we wish to stress the twisted endomorphism, we shall
write~$M^{\alpha}$. The module structure of~$M^{\alpha}$ and the maps
$\alpha^{(n)}$ are related by the square diagram~\eqref{diag:alpha}
below;
\begin{diagram}[textflow,LaTeXeqno]
  \label{diag:alpha}
  M \tensor_{\Rd {0}} \Rd{n} & \rTo^{\alpha^{(n)}} &
  M \tensor_{\Rd{0}} \Rd{0} \\
  \dTo<{\iso} && \dTo<{\iso} \\
  M^{\alpha} \tensor_{\Rp} \Rpn{n} & \rTo[l>=4em] & M^{\alpha}
  \tensor_{\Rp} \Rp
\end{diagram}
the vertical maps, induced by inclusions, are isomorphisms of
$\Rd{0}$\nbd-modules (\cite[Proposition~I.2.12]{P1}), and the bottom
horizontal map is the obvious one, mapping
$x \tensor r \in M^{\alpha} \tensor_{\Rp} \Rpn{n}$ to
$x \tensor r \in M^{\alpha} \tensor_{\Rp} \Rp$. The diagram commutes.

\subsection*{The characteristic sequence of a twisted endomorphism}

For a strongly $\bZ$\nbd-graded ring~$\R$ we have the equality
$\Rd{1} \Rd{-1} = \Rd{0}$, by definition of strongly graded. In
particular, there exist finitely many elements
$x_{j}^{(1)} \in \Rd {1}$ and $y_{j}^{(-1)} \in \Rd {-1}$ so that
$\sum_{j} x_{j}^{(1)} y_{j}^{(-1)} = 1$. For any $\Rp$\nbd-module~$M$
there results a well-defined map of $\Rp$\nbd-modules
\begin{equation}
  \label{eq:tau}
  \chi_{M} = \chi \colon M \tensor_{\Rd {0}} \Rpn {1} \rTo M
  \tensor_{\Rd {0}} \Rp \ , \quad m \tensor r \mapsto m \tensor r -
  \sum_{j} m x_{j}^{(1)} \tensor y_{j}^{(-1)} r
\end{equation}
which does not depend on the specific choice of elements $x_{j}^{(1)}$
and $y_{j}^{(-1)}$ \cite[\S6]{more_on_fd}.

\begin{proposition}[Characteristic sequence of a twised endomorphism]
  \label{prop:characteristic}
  Let $(M, \alpha)$ be an object of~$\twend$. There is a short exact
  sequence of $\Rp$\nbd-modules, natural in $(M,\alpha)$,
  \begin{equation}
    \label{eq:characteristic}
    0 \rTo M^{\alpha} \tensor_{\Rd{0}} \Rpn {1} \rTo^{\chi_{M}} %
    M^{\alpha} \tensor_{\Rd{0}} \Rp \rTo^{\pi} M^{\alpha} \rTo 0 \ ,
  \end{equation}
  with $\pi(m \tensor r) = mr$. The sequence is split exact as a
  sequence of $\Rd{0}$\nbd-modules.
\end{proposition}

This generalises the usual characteristic sequence of a module
equipped with an endomorphism \cite[Proposition~XII.1.1]{MR0249491}.

\begin{proof}
  The sequence~\eqref{eq:characteristic} is nothing but the
  ``canonical resolution'' of \cite[Lemma~6.2]{more_on_fd} for the
  $\Rp$\nbd-module~$M^{\alpha}$ associated with~$(M,\alpha)$.
\end{proof}

\subsection*{Chain complexes of twisted endomorphisms}

A bounded chain complex in the category~$\twend$, that is, an object
of the category $\ch \twend$, consists of a pair $(C, \alpha)$ where
$C$ is a bounded chain complex of $\Rd{0}$\nbd-modules, and
\begin{displaymath} \alpha \colon C \tensor_{\Rd{0}} \Rd{1} %
  \rTo C \tensor_{\Rd{0}} \Rd{0}
\end{displaymath}
is a map of $\Rd{0}$\nbd-module complexes.

\begin{definition}
  The \textit{mapping half-torus of $(C, \alpha) \in \ch \twend$} is
  the complex of $\Rp$\nbd-modules
  \begin{displaymath}
    \torus(C,\alpha) = \cone \Big( \chi_{C} \colon %
    C \tensor_{\Rd{0}} \Rpn{1} %
    \rTo C \tensor_{\Rd{0}} \Rp \Big) \ ,
  \end{displaymath}
  with $\chi_{C} = \chi$ defined in~(\ref{eq:tau}).
\end{definition}

The mapping half-torus provides an exact additive functor defined on the
category $\twend$ to the category of chain complexes of
$\Rp$\nbd-modules. It comes equipped with a natural $\Rp$\nbd-linear
map $\torus(C,\alpha) \rTo C^{\alpha}$ induced from the short exact
sequence~\eqref{eq:characteristic}.

\begin{proposition}[{\cite[Corollaries~6.5 and~6.8]{more_on_fd}}]
  \label{prop:R0-Rp-fd}
  Suppose that $(C,\alpha)$ is a bounded chain complex in~$\twend$.
  \begin{enumerate}
  \item The map $\torus(C,\alpha) \rTo C^{\alpha}$ is a
    quasi-iso\-mor\-phism.
  \item If $C$ is a complex of projective $\Rd{0}$\nbd-modules, the
    map $\torus(C,\alpha) \rTo C^{\alpha}$ is a chain homotopy
    equivalence of $\Rd{0}$\nbd-module complexes.
  \item If $C$ is an $\Rd{0}$\nbd-finitely dominated complex of
    projective $\Rp$\nbd-modules, then $\torus(C,\alpha)$ is
    $\Rp$\nbd-finitely dominated. \qed
  \end{enumerate}
\end{proposition}

\subsection*{Homotopy nilpotent twisted endomorphisms}

Let $(C, \alpha)$ be a chain complex in the category~$\twend$. We say
that $\alpha$ is \textit{homotopy nilpotent} if the chain map
\begin{equation}
  \label{eq:alpha-n}
  \alpha^{(n)} \colon C \tensor_{\Rd{0}} \Rd{n} \rTo %
  C \tensor_{\Rd{0}} \Rd{0}
\end{equation}
is null homotopic for some (equivalently, all) sufficiently large
$n \geq 0$. Using diagram~\eqref{diag:alpha}, this amounts to saying
that the map of $\Rd{0}$\nbd-module complexes
\begin{displaymath}
  C^{\alpha} \tensor_{\Rp} \Rpn{n} \rTo C^{\alpha} \tensor_{\Rp} \Rp = C^{\alpha}
\end{displaymath}
(induced by the inclusion map $\Rpn{n} \subseteq \Rp$) is null
homotopic for $n \gg 0$. Since $\Rpn {n}$ is an invertible
$\Rp$\nbd-bimodule with inverse~$\Rpn{-n}$, this in turn is equivalent
to
\begin{equation}
  \label{eq:null}
  C^{\alpha} = C^{\alpha} \tensor_{\Rp} \Rp \rTo C^{\alpha} \tensor_{\Rp} \Rpn{-n}
\end{equation}
being null homotopic for $n \gg 0$. Under suitable finiteness
assumptions, this is equivalent to the stronger condition that the
map~\eqref{eq:null} is null homotopic for $n \gg 0$ as a map of
$\Rp$\nbd-module complexes:

\begin{lemma}
  \label{lem:middle-acyclic}
  Let $C$ be an $\Rp$\nbd-finitely dominated complex of
  $\Rp$\nbd-modules. The following statements are equivalent:
  \begin{enumerate}
  \item The induced complex $C \tensor_{\Rp} \R$ is acyclic.
  \item For $n \gg 0$ the obvious map
    $\nu_{0,n} \colon C \tensor_{\Rp} \Rpn{0} \rTo C \tensor_{\Rp}
    \Rpn {-n}$,
    considered as a map of $\Rp$\nbd-module chain complexes, is null
    homotopic.
  \item For $n \gg 0$ the obvious map
    $\nu_{0,n} \colon C \tensor_{\Rp} \Rpn{0} \rTo C \tensor_{\Rp}
    \Rpn {-n}$,
    considered as a map of $\Rd{0}$\nbd-module chain complexes, is
    null homotopic.
  \end{enumerate}
  If these equivalent conditions hold, the complex~$C$ is
  $\Rd{0}$-finitely dominated.
\end{lemma}

\begin{proof}
  All three statements and the conclusion on finite domination are
  invariant under homotopy equivalence of complexes of
  $\Rp$\nbd-modules. As $C$ is $\Rp$\nbd-finitely dominated we may,
  and will, assume from the outset that $C$ is a bounded complex of
  finitely generated projective $\Rp$\nbd-modules.

  \smallbreak

  We use the notation
  \begin{gather*}
    \nu \colon C \tensor_{\Rp} \Rp \rTo C \tensor_{\Rp} \R \\
    \intertext{and} \nu_{k,n} \colon C \tensor_{\Rp} \Rpn {-k} \rTo C
    \tensor_{\Rp}
    \Rpn {-k-n} \\
  \end{gather*}
  for the obvious $\Rp$\nbd-linear maps sending $x \tensor r$
  to~$x \tensor r$.  

  \smallbreak

  (1)~$\Rightarrow$~(2): The map~$\nu$ is certainly null homotopic as
  its target is contractible by hypothesis, and a choice of null
  homotopy yields a map
  \begin{displaymath}
    H \colon \operatorname{cyl} (\id_{C}) \rTo C
    \tensor_{\Rp} \R = \bigcup_{n \geq 0} C \tensor_{\Rp}
    \Rpn {-n}
  \end{displaymath}
  restricting to~$\nu$ and~$0$, respectively, on the ``ends'' of the
  cylinder. Since the source is a bounded complex of finitely
  generated projective modules, the map~$H$ factors through a finite
  stage of the increasing union, resulting in a number $N \geq 0$ and
  a map
  \begin{equation}
    \label{eq:Hprime-null}
    H' \colon \operatorname{cyl} (\id_{C}) \rTo %
    C \tensor_{\Rp} \Rpn {-N} \ .
  \end{equation}
  By construction, $H'$ is a null homotopy of $\nu_{0,N}$, and hence
  $\nu_{0,n} \simeq 0$ for all $n \geq N$.

  \smallbreak

  (2)~$\Rightarrow$~(3): This is a tautology.

  \smallbreak

  (3)~$\Rightarrow$~(1): Let $n$ be sufficiently large so that
  $\nu_{0,n} \simeq 0$ as a map of $\Rd{0}$\nbd-module complexes. As
  $\nu_{k,n}$ is isomorphic to $\nu_{0,n} \tensor_{\Rd{0}} \Rd{-k}$,
  as a map of $\Rd{0}$\nbd-modules (\cite[Proposition~I.2.12]{P1}), we
  conclude that $\nu_{k,n} \simeq 0$ for all $k \geq 0$. In
  particular, the maps $\nu_{k,n}$ induce the trivial map on
  homology. Since
  $C \tensor_{\Rp} \R = \dirlim_{n \geq 0} C \tensor_{\Rp} \Rpn{-n}$,
  and since homology commutes with filtered colimits, we conclude by
  cofinality that
  \begin{displaymath}
    H_{*} (C \tensor_{\Rp} \R) = \dirlim_{n \geq 0}
  H_* (C \tensor_{\Rp} \Rpn{-n}) = \dirlim_{k \geq 0} H_*
  (C \tensor_{\Rp} \Rpn{-kn}) \ ,
  \end{displaymath}
  the last colimit taken with respect to the maps $H_{*}(\nu_{kn,n})$
  which are trivial. It follows that the colimit is trivial, whence
  $C \tensor_{\Rp} \R$ is acyclic.

  \medbreak

  Suppose now that the equivalent conditions of the Lemma are
  satisfied. Going back to~\eqref{eq:Hprime-null} and the notation
  used there, we know that the chain map
  \begin{displaymath}
    \nu_{0,N} \colon C \iso C \tensor_{\Rp} \Rp \rTo C \tensor_{\Rp}
    \Rpn {-N}
  \end{displaymath}
  is null homotopic. Hence its mapping cone is homotopy equivalent to
  the mapping cone
  $M = C \oplus \Sigma \big( C \tensor_{\Rp} \Rpn{-N} \big)$ of the
  zero map between the same complexes, which contains $C$ as a direct
  summand. On the other hand, the map $\nu_{0,N}$ is injective (since
  $C$ is assumed to consist of projective $\Rp$\nbd-modules), hence
  its mapping cone is quasi-isomorphic to the cokernel~$K$
  of~$\nu_{0,N}$. The $\ell$th chain module $K_{\ell}$ of~$K$ is
  isomorphic, as an $\Rp$\nbd-module, to
  $C_{\ell} \tensor_{\Rp} (\Rpn{-N} / \Rp)$. Since $C_{\ell}$ is a
  direct summand of 
  $\Rp^{n_{\ell}}$, for suitable $n_{\ell} \geq 0$, the module
  $K_{\ell}$ is a direct summand of
  \begin{multline*}
    \Rp^{n_{\ell}} \tensor_{\Rp} (\Rpn{-N} / \Rp) \iso %
    \Big( \Rp \tensor_{\Rp} (\Rpn{-N} / \Rp) \Big)^{n_{\ell}} \\%
    \iso \big( \Rpn{-N} / \Rp \big)^{n_{\ell}} %
    \iso \Big( \bigoplus_{q=-N}^{-1} \Rd{q} \Big)^{n_{\ell}} \ ,
  \end{multline*}
  the last isomorphism being $\Rd{0}$\nbd-linear. Since $\R$ is
  strongly graded each $\Rd{q}$ is a finitely generated projective
  $\Rd{0}$\nbd-module so that $K_{\ell}$ is a finitely generated
  projective $\Rd{0}$\nbd-module as well. Thus the (bounded) complex
  $K$ is $\Rd{0}$\nbd-finitely dominated. Since $C$ consists of
  projective $\Rd{0}$\nbd-modules so does~$M$, and as $M$ is
  quasi-isomorphic to~$K$ by the above arguments there is in fact an
  $\Rd{0}$\nbd-linear homotopy equivalence $M \simeq K$. It follows
  that $M$ is $\Rd{0}$\nbd-finitely dominated, hence so is its direct
  summand~$C$. --- The finiteness result also follows from Theorem~8.1
  of~\cite{more_on_fd}; in the notation used there the ring
  $\R_{*}(\!(t\inv)\!)$ contains~$\R$ so that
  $C \tensor_{\Rp} \R_{*}(\!(t\inv)\!) \iso C \tensor_{\Rp} \R
  \tensor_{\R} \R_{*}(\!(t\inv)\!)$ is acyclic as required.
\end{proof}

From Lemma~\ref{lem:middle-acyclic}, and from the fact that the
maps~\eqref{eq:alpha-n} and~\eqref{eq:null} are isomorphic as
$\Rd{0}$\nbd-linear maps, we conclude:

\begin{corollary}
  \label{cor:nilpotent_iff_acyclic}
  Let $(C, \alpha)$ be an object of~$\ch \twend$. Suppose that the
  associated $\Rp$\nbd-module complex $C^{\alpha}$ is
  $\Rp$\nbd-finitely dominated. The twisted endomorphism $\alpha$ is
  homotopy nilpotent if and only if $C^{\alpha} \tensor_{\Rp} \R$ is
  acyclic. \qed
\end{corollary}

For later use we record the following useful fact:

\begin{lemma}
  \label{lem:R-acyclyc-R0-fd}
  Let $(Z, \zeta)$ be an object of~$\ch \twend$. Suppose that the
  associated $\Rp$\nbd-module complex $Z^{\zeta}$ is an
  $\Rp$\nbd-finitely dominated complex of projective
  $\Rp$\nbd-modules. Suppose further that $Z^{\zeta} \tensor_{\Rp} \R$
  is acyclic. Then $Z$ is an $\Rd{0}$\nbd-finitely dominated bounded
  complex of projective $\Rd{0}$\nbd-modules.
\end{lemma}

\begin{proof}
  Since $\R$ is \textit{strongly} $\bZ$\nbd-graded the complex~$Z$
  consists of projective $\Rd{0}$\nbd-modules. By hypothesis,
  $Z^\zeta$ is homotopy equivalent to a bounded complex~$C$ of
  finitely generated projective $\Rp$\nbd-modules. Since
  $Z^{\zeta} \tensor_{\Rp} \R$ is acyclic so is $C \tensor_{\Rp} \R$.
  This forces~$C$,
  hence~$Z^{\zeta}$, to be $\Rd{0}$\nbd-finitely dominated in view of
  Lemma~\ref{lem:middle-acyclic}.
\end{proof}

\subsection*{The nil category}

We are now in a position to define the nil category, the category of
homotopy nilpotent endomorphisms of chain complexes satisfying a
suitable finiteness constraint. In fact there are two variants,
corresponding to the two subrings $\Rp$ and~$\Rn$ of~$\R$.

\begin{definition}
  The \textit{positive nil category~$\cT{+}$ of~$\Rd{0}$ relative
    to~$\R$} is the full subcategory of $\ch (\twend)$ consisting of
  those chain complexes $(C, \alpha)$ such that
  \begin{enumerate}
  \item the underlying $\Rd{0}$\nbd-module complex~$C$ consists of
    projective $\Rd{0}$\nbd-modules, and is $\Rd{0}$\nbd-finitely
    dominated;
  \item the chain map
    $\alpha \colon C \tensor_{\Rd{0}} \Rd{1} \rTo C \tensor_{\Rd{0}}
    \Rd{0}$
    is homotopy nilpotent in the sense that $\alpha^{(n)} \simeq 0$
    for $n \gg 0$.
  \end{enumerate}
  The \textit{negative nil category~$\cT{-}$ of~$\Rd{0}$ relative
    to~$\R$} is defined analogously, using the category $\twendm$ of
  negative twisted endomorphisms in place of $\twend$.
\end{definition}

\begin{remark}
  \label{rem:nil-cat}
  Let $(C,\alpha)$ be an object of~$\cT{+}$. By
  Proposition~\ref{prop:R0-Rp-fd}, the $\Rp$\nbd-linear map
  $\torus(C,\alpha) \rTo C^{\alpha}$ is a quasi-isomorphism with
  $\Rp$\nbd-finitely dominated source. The same map is an
  $\Rd{0}$\nbd-linear homotopy
  equivalence.
  Thus the second condition in the definition of $\cT{+}$ is
  equivalent to the condition
  $\torus(C,\alpha) \tensor_{\Rp} \R \simeq 0$, by
  Corollary~\ref{cor:nilpotent_iff_acyclic}.
\end{remark}

A morphism in the category $\cT{+}$ is called a \textit{cofibration} if
it is injective and its cokernel consists of projective
$\Rd{0}$\nbd-modules. We say the morphism is a \textit{weak
  equivalence}, or a \textit{$q$\nbd-equivalence}, if it is a
quasi-isomorphism of $\Rd{0}$\nbd-module complexes.

\begin{lemma}
  These definitions equip $\cT{+}$ with the structure of a category with
  cofibrations and weak equivalences.
\end{lemma}

\begin{proof}
  This is mostly straightforward; the gluing lemma holds, for example,
  since it holds in the category of bounded complexes of projective
  $\Rd{0}$\nbd-modules. What needs explicit verification is that the
  requisite pushouts exist within the category~$\cT{+}$.  Using the
  functor~$\Phi$ and its inverse $(A,\alpha) \mapsto A^{\alpha}$ we
  identify $\cT{+}$ with a full subcategory of the category~$\cK$ of
  bounded chain complexes of $\Rp$\nbd-modules. Let $(A,\alpha)$,
  $(B,\beta)$ and $(C,\gamma)$ be objects in~$\cT{+}$, let
  $f \colon (A,\alpha) \rTo (B,\beta)$ be a cofibration, and let
  $g \colon (A,\alpha) \rTo (C,\gamma)$ be an arbitrary morphism
  in~$\cT{+}$. We can then form the pushout diagram
  \begin{diagram}
    A^{\alpha} & \rTo^{f} & B^{\beta} \\
    \dTo<{g} & & \dTo \\
    C^{\gamma} & \rTo^{f'} & P^{\pi}
  \end{diagram}
  in the category~$\cK$, and claim that $(P, \pi)$ is on object
  of~$\ch(\twend)$.

  As $f$ is a cofibration in~$\cT{+}$ it is an injective map, and the
  same is thus true for its pushout~$f'$. Let $K^{\kappa}$ denote the
  cokernel of~$f$, associated with $(K, \kappa) \in \ch(\twend)$. From
  general properties of pushout squares we know that
  $\coker(f') \iso K^{\kappa}$. Hence we obtain a commutative ladder
  diagram of short exact sequences of chain complexes:
  \begin{diagram}[flow]
    0 & \rTo & A^{\alpha} & \rTo & B^{\beta} & \rTo & K^{\kappa} &
    \rTo & 0 \\
    && \dTo && \dTo && \dTo>{\iso} \\
    0 & \rTo & C^{\gamma} & \rTo & P^{\pi} & \rTo & K^{\kappa} &
    \rTo & 0
  \end{diagram}
  Again, since $f$ is a cofibration its cokernel $K^{\kappa}$ consists
  of projective $\Rd{0}$\nbd-modules. Thus both sequences 
  are levelwise split short exact sequences of $\Rd{0}$\nbd-module
  complexes, and $P$ consists of projective $\Rd{0}$\nbd-modules. By
  Lemma~\ref{lem:K-fd-2-of-3}, applied to the top row, the
  complex~$K^\kappa$ is $\Rd{0}$\nbd-finitely dominated, which
  implies, by applying Lemma~\ref{lem:K-fd-2-of-3} to the bottom row,
  that so is~$P^{\pi}$.

  Application of the exact half-torus function yields another
  commutative ladder diagram of short exact sequences,
  \begin{diagram}[flow]
    0 & \rTo & \torus(A,\alpha) & \rTo & \torus(B,\beta) & \rTo &
    \torus(K,\kappa) & \rTo & 0 \\
    && \dTo && \dTo && \dTo>{\iso} \\
    0 & \rTo & \torus(C,\gamma) & \rTo & \torus(P,\pi) & \rTo &
    \torus(K,\kappa) & \rTo & 0
  \end{diagram}
  consisting of bounded complexes of finitely generated projective
  $\Rp$\nbd-modules; in particular, both short exact sequences are
  levelwise split. By Remark~\ref{rem:nil-cat} the chain complexes
  $\torus(A,\alpha)$, $\torus(B,\beta)$ and~$\torus(C,\gamma)$ are
  $\Rp$\nbd-finitely dominated. Hence $\torus(K,\kappa)$ and
  $\torus(P,\pi)$ are $\Rp$\nbd-finitely dominated as well, by two
  applications of Lemma~\ref{lem:K-fd-2-of-3}.

  The bottom row yields a short exact sequence
  \begin{displaymath}
    0 \rTo \torus(C,\gamma) \tensor_{\Rp} \R \rTo \torus(P,\pi)
    \tensor_{\Rp} \R \rTo \torus(K,\kappa) \tensor_{\Rp} \R \rTo 0 \ .
  \end{displaymath}
  The first and third entry are acyclic, by
  Corollary~\ref{cor:nilpotent_iff_acyclic}, hence so is the middle
  entry. Applying the Corollary again leads us to conclude that
  $(P,\pi)$ is an object of~$\cT{+}$.
\end{proof}

With this lemma in place, we can now introduce the notation
\begin{subequations}
  \begin{align}
    \nil+q(\Rd{0}) & = K_{q} \big( \cT{+} \big) = %
                     \pi_{q} \Omega |q \sdot \cT{+}| %
    \label{eq:notation_nil+} \\ %
    \intertext{and}
    \nil-q(\Rd{0}) & = K_{q} \big( \cT{-} \big) = %
                     \pi_{q} \Omega |q \sdot \cT{-}| \ . %
    \label{eq:notation_nil-} %
  \end{align}
\end{subequations}

\newcommand{\fdr}{\mathbf{FD}}

For $L$ a unital ring write $\fdr(L)$ for the category of
$L$\nbd-finitely dominated bounded complexes of projective
$L$\nbd-modules. It is well-known, and can be verified without
difficulty using \textsc{Waldhausen}'s approximation theorem, that the
inclusion $\chpr L \rTo^{\subseteq} \fdr(L)$ induces a homotopy
equivalence on $K$\nbd-theory spaces
\begin{equation}
  \label{eq:FD}
  K(L) = \Omega |q\sdot \chpr L| \rTo^{\simeq} \Omega |q \sdot
  \fdr(L)| \ ,
\end{equation}
where ``$q$'' stands for quasi-isomorphisms as usual, and the
cofibrations are the injective maps with levelwise projective
cokernel. Hence the forgetful functors
\begin{displaymath}
  o^{\mp} \colon \cT{\mp} \rTo \fdr(\Rd{0}) \ , \quad (Z, \zeta)
  \mapsto Z
\end{displaymath}
yield group homomorphisms
\begin{equation}
  \label{eq:maps-o}
  o^{-}_{*} \colon \nil-q(\Rd{0}) \rTo K_{q} (\Rd{0}) %
  \qquad \text{and} \qquad %
  o^{+}_{*} \colon \nil+q(\Rd{0}) \rTo K_{q} (\Rd{0}) \ ,
\end{equation}
and we will establish in Theorem~\ref{thm:NK_is_nil} that there are
isomorphisms
\begin{displaymath}
  \NK-{q+1}(\Rd{0}) \iso \ker(o^{+}) %
  \qquad \text{and} \qquad %
  \NK+{q+1}(\Rd{0}) \iso \ker(o^{-}) \ , %
\end{displaymath}
identifying the groups $\NK{\mp}{q+1} (\Rd{0})$ with the $q$th reduced
algebraic $K$\nbd-group of the category~$\cT{\pm}$.

\section{The fundamental square} 

\subsection{The fundamental square of the projective line}

Let $f = (f^{-}, f^{0}, f^{+})$ be a morphism $\cY \rTo \cZ$ in the
category~$\cC$ of bounded chain complexes of vector bundles on the
projective line associated with a strongly $\bZ$-graded ring~$\R$. We
say that $f$ is an \textit{$h_{?}$\nbd-equivalence}, for the
decoration $? \in \{-,\, 0,\, +\}$, if the component~$f^{?}$ is a
quasi-isomorphism of chain complexes. Together with the previous
notion of cofibrations, \viz, levelwise split injections in each
component, this equips $\cC$ with three new structures of a category
with cofibrations and weak equivalences. Note that an
$h^{-}$\nbd-equivalence is automatically an $h^{0}$\nbd-equivalence as
well since in the commutative square
\begin{diagram}
  Y^{-} \tensor_{\Rn} \R & \rTo^{\upsilon^{-}_{\sharp}} & Y^{0} \\
  \dTo<{f^{-} \tensor \R} && \dTo<{f^{0}} \\
  Z^{-} \tensor_{\Rn} \R & \rTo^{\zeta^{-}_{\sharp}} & Z^{0}
\end{diagram}
the horizontal maps are isomorphisms (sheaf condition), and the left
vertical map is a quasi-isomorphism because $\R$ is a flat left
$\Rn$\nbd-module thanks to the strong grading. --- Similarly, every
$h^{+}$\nbd-equivalence is an $h^{0}$\nbd-equivalence.

Thus the identity functor yields a commutative square of
$K$\nbd-theory spaces
\begin{equation}
  \begin{diagram}
    h \mathcal{S}_{\bullet} \cC & \rTo & h_{+} \mathcal{S}_{\bullet} \cC \\
    \dTo && \dTo \\
    h_{-} \mathcal{S}_{\bullet} \cC & \rTo & h_{0}
    \mathcal{S}_{\bullet} \cC
  \end{diagram}
  \label{eq:fund}
\end{equation}
which we call the \textit{fundamental square} of the projective line.
Associated with it, by taking vertical homotopy fibres, is the map
\begin{equation}
  \label{eq:map_alpha}
  \alpha \colon h \mathcal{S}_{\bullet} \cC^{h_{-}} %
  \rTo h_{+} \mathcal{S}_{\bullet} \cC^{h_{0}}
\end{equation}
where the notation is as in \textsc{Waldhausen}'s fibration theorem
\cite[Theorem~1.6.4]{MR802796}. Explicitly, $\cC^{h_{-}}$ is the full
subcategory of~$\cC$ consisting of those objects $\cY$ for which
$\cY \rTo 0$ is an $h_{-}$\nbd-equivalence, that is, which satisfy
$Y^{-} \simeq *$ and $Y^{0} \simeq *$, and $\cC^{h_{0}}$ is the full
subcategory of~$\cC$ consisting of those objects $\cY$ for which
$\cY \rTo 0$ is an $h_{0}$\nbd-equivalence, that is, which satisfy
$Y^{0} \simeq *$.

\subsection*{The fibres of the fundamental square}

The (homotopy) fibres of the maps in the fundamental square can be
identified explicitly: they are homotopy equivalent to $K$\nbd-theory
spaces of categories of homotopy nilpotent twisted endomorphisms. We
will use this identification presently to conclude that the
fundamental square is homotopy cartesian.

\begin{theorem}[Fibres of the fundamental square]
  \label{thm:fibre-terms-are-nil}
  \addtocounter{equation}{-1}
  \begin{subequations}
    The functor
    \begin{displaymath}
      F \colon \ch \vb^{h_{0}} \rTo \cT{+} \ , \quad \cY = \big( Y^{-} \rTo
      Y^{0} \lTo Y^{+} \big) \ \mapsto \ \Phi (Y^{+}) \ ,
    \end{displaymath}
    defined on the category of bounded chain complexes of vector
    bundles~$\cY$ with $Y^{0} \simeq *$, induces a homotopy equivalence
    \begin{align}
      h_{+} \sdot \ch \vb^{h_{0}} \rTo^{\sim} q \sdot \cT{+}  %
      \ , \label{map:h+} \\ %
      \intertext{and, by restriction, a homotopy equivalence} %
      h \sdot \ch \vb^{h_{-}} \rTo^{\sim} q \sdot \cT{+} %
      \label{map:h}
    \end{align}
    where the letter ``$q$'' stands for weak equivalences in the
    category~$\cT{+}$.
  \end{subequations}
\end{theorem}

By symmetry, there are analogous homotopy equivalences
\begin{subequations}
  \begin{align}
    h_{-} \sdot \ch \vb^{h_{0}} \rTo^{\sim} q \sdot \cT{-}  %
    \ , \label{map:h-} \\ %
    \intertext{and, by restriction,} %
    h \sdot \ch \vb^{h_{+}} \rTo^{\sim} q \sdot \cT{-} \ .%
    \label{map:h_alt}
  \end{align}
\end{subequations}

\begin{proof}
  The functor $F$ is well defined. Indeed, the component $Y^{+}$
  of~$\cY$ is a bounded complex of finitely generated projective
  $\Rp$\nbd-modules, hence $Y^{+}$ consists of projective
  $\Rd{0}$\nbd-modules. The hypothesis
  $Y^{+} \tensor_{\Rp} \R \iso Y^{0} \simeq *$ implies that $Y^{+}$ is
  $\Rd{0}$\nbd-finitely dominated by
  Lemma~\ref{lem:R-acyclyc-R0-fd}. Moreover, the associated twisted
  endomorphism
  \[Y^{+} \tensor_{\Rd{0}} \Rd{1} \rTo Y^{+} \tensor_{\Rd{0}} \Rd{0}\]
  is homotopy nilpotent by Corollary~\ref{cor:nilpotent_iff_acyclic}.

  \smallbreak

  We will employ \textsc{Waldhausen}'s approximation theorem
  \cite[Theorem~1.6.7]{MR802796}, combined with the standard
  observation that the map required by property (App~2) does not have
  to be a cofibration (since it can be replaced by one in the presence
  of cylinder functors). --- We can in fact treat cases~\eqref{map:h+}
  and~\eqref{map:h} at the same time as the proofs are almost
  identical.

  \smallbreak To start with, a morphism $a=(a^{-}, a^{0}, a^{+})$ in
  $\ch \vb^{h_{0}}$ is an $h_{+}$\nbd-equi\-va\-lence, by definition, if
  and only if $F(a)=a^{+}$ is a quasi-isomorphism. A morphism
  $b=(b^{-}, b^{0}, b^{+})$ in $\ch \vb^{h_{-}}$ is an
  $h$\nbd-equivalence if and only if $F(b)=b^{+}$ is a
  quasi-isomorphism since $b^{-}$ and~$b^{0}$ are maps between acyclic
  complexes, hence are quasi-isomorphisms in any case. In other words,
  the functor~$F$ satisfies property~(App~1) in both cases under
  consideration.

  \smallbreak

  Let $(Z, \zeta) \in \cT{+}$ and $Y^{+} \in \ch \vb^{h_{0}}$, and let
  $f \colon \Phi(Y^{+}) \rTo (Z, \zeta)$ be a morphism in~$\cT{+}$. We
  can equivalently consider $f$ as an $\Rp$\nbd-linear map of chain
  complexes $Y^{+} \rTo Z^{\zeta}$. By Remark~\ref{rem:nil-cat}, the
  chain complex $Z^{\zeta}$ is quasi-isomorphic to the
  $\Rp$\nbd-finitely dominated complex~$\torus(Z,\zeta)$ so that there
  exists a bounded complex~$E$ of finitely generated projective
  $\Rp$\nbd-modules and a quasi-isomorphism
  $e \colon E \rTo Z^{\zeta}$. We can lift~$f$ up to homotopy to an
  $\Rp$\nbd-linear map $g \colon Y^{+} \rTo E$ so that $eg$ is
  homotopic to~$f$. A choice of homotopy $f \simeq eg$ determines a
  map $h' \colon \cyl(g) \rTo Z^{\zeta}$ such that the composite
  $Y^{+} \rTo_{i_{0}} \cyl(g) \rTo_{h'} Z^{\zeta}$ coincides with~$f$,
  and such that the composite
  $E \rTo^{\simeq}_{i_{1}} \cyl(g) \rTo_{h'} Z^{\zeta}$ coincides
  with~$e$. In particular, $h'$~is a quasi-isomorphism of
  $\Rp$\nbd-module complexes. (Here $i_{0}$ and~$i_{1}$ are the front
  and back inclusion into the mapping cylinder.)

  We can now form a complex~$X^{+}$ by attaching to~$\cyl(g)$ a direct
  sum of contractible complexes of the type $D(\ell,M)$, with $M$ a
  finitely generated projective $\Rp$\nbd-module, so that all chain
  modules of~$X^{+}$ except possibly one are finitely generated
  \textit{free} modules. We let $a^{+}$ denote the composition
  of~$i_{0}$ with the inclusion $\cyl(g) \rTo X^{+}$, and let $h^{+}$
  denote the composition of the projection $X^{+} \rTo \cyl(g)$
  with~$h'$. We have constructed a commutative diagram
  \begin{diagram}
    Y^{+} & \rTo^{a^{+}} & X^{+} \\
    & \rdTo<{f} & \dTo<{\simeq}>{h^{+}} \\
    && Z^{\zeta}
  \end{diagram}
  where $X^{+}$ is a bounded complex of finitely generated projective
  $\Rp$\nbd-modules, and $h^{+}$ is a quasi-isomorphism.

  \smallbreak

  We will now extend $X^{+}$ to a vector bundle
  $\cX = (X^{-}, X^{0}, X^{+})$ and $a^{+}$ to a map of vector bundles
  $a = (a^{-}, a^{0}, a^{+}) \colon \cY \rTo \cX$. This may involve
  enlarging $X^{+}$ further, by taking the direct sum with
  contractible complexes of the form $D(k,\Rp^{k})$ and composing
  $a^{+}$ and $h^{+}$ with the obvious inclusion and projection
  maps. We shall keep the notation $X^{+}$ and $h^{+}$ even for the
  modified data.

  \smallbreak

  Let $X^{0} = X^{+} \tensor_{\Rp} \R$. Recall that $Z^{\zeta}$ is
  quasi-isomorphic to~$X^{+}$, \textit{via} the map~$h^{+}$, and to
  the half-torus~$\torus(Z,\zeta)$, by Remark~\ref{rem:nil-cat}. Both
  $X^{+}$ and~$\torus(Z,\zeta)$ are bounded complexes of finitely
  generated projective $\Rp$\nbd-modules, so these complexes are
  homotopy equivalent. It follows that
  $X^{0} \simeq \torus(Z,\zeta) \tensor_{\Rp} \R$ is contractible, by
  Remark~\ref{rem:nil-cat} again; this implies that
  $[X^{0}] = \sum_{k} (-1)^{k} [X^{0}_{k}] = 0 \in K_{0} (\R)$. As all
  chain modules of $X^{0}$ are free with the possible exception of a
  single module, we conclude that $X^{0}$ consists of stably free
  modules, and hence consists of modules which are stably induced
  from~$\Rn$. We now appeal to
  Theorem~\ref{thm:lifting_acyclic_complexes}: we can modify $X^{+}$
  by taking direct sum with finitely many contractible complexes of
  the form $D(k, \Rp^{j})$, thereby modifying $X^{0}$ in an analogous
  manner, so that $X^{0}$ is isomorphic to $X^{-} \tensor_{\Rn} \R$
  for a bounded acyclic complex~$X^{-}$ of finitely generated
  projective $\Rn$\nbd-modules.

  \smallbreak

  We have thus constructed a vector bundle
  \begin{displaymath}
    \cX = (X^{-} \rTo X^{0} \lTo X^{+})
  \end{displaymath}
  together with a quasi-isomorphism
  $h^{+} \colon X^{+} \rTo Z^{\zeta}$ and a map
  $a^{+} \colon Y^{+} \rTo X^{+}$ such that $h^{+} \circ a^{+} = f$.
  The map $a^{+}$ induces a compatible map
  $a^{0} = \xi^{+}_{\sharp} \circ ( a^{+} \tensor \id) \circ \big(
  \xi^{+}_{\sharp} \big)\inv$,
  see the diagram in Fig.~\ref{fig:diag-corners}.
  \begin{figure}[ht]
    \centering
    \begin{diagram}
      Y^{-} & \rTo^{\upsilon^{-}} & Y^{0} &
      \lTo_{\iso}^{\xi^{+}_{\sharp}} & Y^{+} \tensor_{\Rp} \R \\
      && \dDashto<{a^{0}} && \dTo>{a^{+} \tensor \id} \\
      X^{-} & \rTo^{\xi^{-}} & X^{0} & \lTo_{=}^{\xi^{+}_{\sharp}} &
      X^{+} \tensor_{\Rp} \R
    \end{diagram}
    \caption{Diagram used in proof of
      Theorem~\ref{thm:fibre-terms-are-nil}}
    \label{fig:diag-corners}
  \end{figure}
  For each chain level~$\ell$ the composite map
  \begin{displaymath}
    Y^{-}_{\ell} \rTo^{\upsilon^{-}} Y^{0}_{\ell} \rTo^{a^{0}} X^{0}_{\ell}
    \iso X^{-}_{\ell} \tensor_{\Rn} \R = \bigcup_{q \geq 0}
    X^{-}_{\ell} \tensor_{\Rn} \Rnn{q}
  \end{displaymath}
  factors through some term $X^{-}_{\ell} \tensor_{\Rn} \Rnn{q}$ so
  that, by choosing $q \gg 0$, the map $a^{0} \circ \upsilon^{-}$
  factorises as
  \begin{displaymath}
    Y^{-} \rTo^{a^{-}} X^{-} \tensor_{\Rn} \Rnn{q} \rTo X^{0} \ ,
  \end{displaymath}
  the second map in this composition given by
  $x \tensor r \mapsto \xi^{-}(x) \cdot r$. That is, there exists
  $q \geq 0$ and
  $a^{-} \colon Y^{-} \rTo X^{-} \tensor_{\Rn} \Rnn {q}$ such that
  $a = (a^{-}, a^{0}, a^{+}) \colon \cY \rTo \cX(q,0)$ is a map of
  vector bundles, and such that $h^{+} \circ F(a) = f$.  Since $X^{-}$
  and $X^{0}$ are contractible by construction, $\cX(q,0)$ is an
  object of $\ch \vb^{h_{-}} \subseteq \ch \vb^{h_{0}}$.

  This proves that $F$ satisfies property~(App~2) in addition
  to~(App~1), and the approximation theorem applies.
\end{proof}

\begin{corollary}
  \label{cor:fund_square_cartesian}
  The map~$\alpha$ of~\eqref{eq:map_alpha} is a homotopy
  equivalence. Thus, the fundamental square~\eqref{eq:fund} is
  homotopy cartesian.
\end{corollary}

\begin{proof}
  The previous Theorem asserts that in the chain of maps
  \begin{displaymath}
    h \sdot \ch \vb^{h_{-}} \rTo^{\alpha} h_{+} \sdot \ch \vb^{h_{0}}
    \rTo^{F} q \sdot \cT{+}
  \end{displaymath}
  both $F$ and~$F \circ \alpha$ are homotopy equivalences. It follows
  that $\alpha$ is a homotopy equivalence as well. Since the
  fundamental square consists of connected spaces it is homotopy
  cartesian.
\end{proof}

\subsection*{Auxiliary categories}

We let $\mathbb{C}_{+}$ denote the category of bounded chain
complexes of ``vector bundles on $\operatorname{spec}(\Rp)$'', that
is, diagrams of the form
\[Y^{0} \lTo^{\upsilon^{+}} Y^{ + }\]
with $Y^{0}$ and $Y^{ + }$ being bounded complexes of finitely
generated projective over $\R$ and $\Rp$, respectively, subject to the
condition that the associated
adjoint map $Y^{0} \lTo Y^{ + } \tensor_{\Rp} R$ be
an isomorphism. A \textit{morphism} $g = (g^{0}, g^{+})$ from
$Y^{0} \lTo^{\upsilon^{+}} Y^{ + }$ to
$Z^{0} \lTo^{\zeta^{+}} Z^{ + }$ consists of an $\R$\nbd-linear map
$g^{0} \colon Y^{0} \rTo Z^{0}$ and an $\Rp$\nbd-linear map
$g^{+} \colon Y^{+} \rTo Z^{+}$ such that $\zeta^{+} \circ g^{+} =
g^{0} \circ \upsilon^{+}$.

By $\mathbb{D}_{+}$ we denote the full subcategory of~$\mathbb{C}_{+}$
consisting of objects $Y^{0} \lTo^{\upsilon^{+}} Y^{ + }$ such
that all chain modules $Y^{0}_{k}$ are stably induced from~$\Rn$, that
is, such that
\begin{displaymath}
  [Y^{0}_{k}] \in \im \big( K_{0}(\Rn) \rTo K_{0} (\R)
  \big) \ ,
\end{displaymath}
or equivalently, such that $\alpha \big( [Y^{0}_{k}] \big) = 0$ in
$K_{0} (\Rn \downarrow \R)$ for all~$k$.

\begin{lemma}
  \label{lem:D_stably_extends}
  Every object of~$\mathbb{D}_{+}$ stably extends to an object
  of~$\cC$. That is, given an object
  $Y^{0} \lTo^{\upsilon^{+}} Y^{ + }$ of~$\mathbb{D}_{+}$ there exist
  an object
  \begin{displaymath}
    Z^{-} \rTo^{\zeta^{-}} Z^{0} \lTo^{\zeta^{+}} Z^{+} 
  \end{displaymath}
  of~$\cC$ and finitely many numbers $j_{k} \geq 0$ together with a
  commutative diagram
  \begin{diagram}
    Y^{0} \oplus \bigoplus_{k} D(k, \R^{j_{k}}) %
    & \lTo[l>=5em]^{\upsilon^{+} \oplus \inc} %
    & Y^{+} \oplus \bigoplus_{k} D(k, \Rp^{j_{k}}) \\ %
    \dTo<{=} && \dTo<{=} \\ %
    Z^{0} & \lTo^{\zeta^{+}} & Z^{+} \ ,
  \end{diagram}
  where ``$\inc$'' denotes the obvious inclusion map based on the ring
  inclusion $\R \supseteq \Rp$.
\end{lemma}

\begin{proof}
  We apply Proposition~\ref{prop:make_induced} to the ring
  inclusion $f \colon \Rn \rTo \R$ and the chain complex $D =
  Y^{0}$. We obtain a stabilisation
  \begin{displaymath}
    D' = Y^{0} \oplus \bigoplus_{k} D(k, \R^{j_{k}})
  \end{displaymath}
  of $D = Y^{0}$ and a bounded complex of finitely generated
  projective $\Rn$\nbd-modules $Z^{-} = C'$ such that there is an
  isomorphism $i \colon Z^{-} \tensor_{\Rn} \R \rTo D'$. Write $Z^{0}$
  in place of~$D'$, let $\zeta^{-} \colon Z^{-} \rTo Z^{0}$ be the
  $\Rn$\nbd-linear map adjoint to~$i$, and define
  \begin{displaymath}
    Z^{+} = Y^{+} \oplus \bigoplus_{k} D(k, \Rp^{j_{k}}) \ ;
  \end{displaymath}
  with $\zeta^{+} = \upsilon^{+} \oplus \inc$ the data satisfies all
  the requirements of the Lemma.
\end{proof}

A morphism $g = (g^{0}, g^{+})$ in~$\mathbb{C}_{+}$ is called a
\textit{cofibration} if both $g^{0}$ and~$g^{+}$ are levelwise
injections such that $\coker(g)$ is an object of~$\mathbb{C}_{+}$. We
call $g$ an \textit{$h_{+}$\nbd-equivalence} if $g^{+}$ is a
quasi-isomorphism (and hence so is
$g^{0} \iso g^{+} \tensor_{\Rp} \R$).

\begin{lemma}
  \label{lem:C_D_are_Waldhausen}
  With these definitions, both $\mathbb{C}_{+}$ and $\mathbb{D}_{+}$
  are categories with cofibrations and weak equivalences.
\end{lemma}

\begin{proof}
  This is clear for~$\mathbb{C}_{+}$ since this category is equivalent
  to the category of bounded chain complexes of finitely generated
  projective $\Rp$-modules, via the functors
  \begin{displaymath}
    \big( Y^{0} \lTo Y^{+} \big) \quad \mapsto \quad Y^{+} %
    \qquad \text{and} \qquad %
    M^{+} \quad \mapsto \quad %
    \big( M^{+} \tensor_{\Rp} R \lTo M^{+} \big) \ . %
  \end{displaymath}
  It remains to observe that the cokernel of a cofibration~$g$
  in~$\mathbb{D}_{+}$ is automatically an object of~$\mathbb{D}_{+}$
  since, given $g^{0}_{k} \colon Y^{0}_{k} \rTo Z^{0}_{k}$, we have
  the equality
  \begin{displaymath}
    \alpha \big( [\coker g^{0}_{k}] \big) = \alpha \big( [Z^{0}_{k}]
    \big) - \alpha \big( [Y^{0}_{k}] \big) = 0
  \end{displaymath}
  in the group $K_{0} (\Rn \downarrow \R)$.
\end{proof}

We make the analogous symmetric definitions for $\mathbb{C}_{-}$
and~$\mathbb{D}_{-}$. The category $\mathbb{C}_{0}$ is defined to be
the category of bounded chain complexes of finitely generated
projective $\R$-modules, with $\mathbb{D}_{0}$ the full subcategory of
those objects $Y^{0}$ such that all chain modules $Y^{0}_{k}$ are
stably induced from both $\Rn$ and~$\Rp$, that is, such that
\begin{displaymath} %
  [Y^{0}_{k}] \in \im \big( K_{0} \Rn %
  \rTo K_{0} R \big) %
  \,\cap\,%
  \im \big( K_{0} \Rp \rTo K_{0} R \big) %
  \qquad \text{for all \(k\)}\ .
\end{displaymath}
Lemmas~\ref{lem:D_stably_extends} and~\ref{lem:C_D_are_Waldhausen} are
valid \textit{mutatis mutandis} for these categories.

\begin{lemma}
  \label{lem:correct_K0}
  The forgetful functor
  \begin{displaymath}
    \big( Y^{0} \lTo Y^{+} \big) \quad \mapsto \quad Y^{+}
  \end{displaymath}
  induces isomorphisms on algebraic $K$-groups
  \[K_{q} \mathbb{D}_{+} \rTo^{\iso} K_{q} \Rp \qquad (q > 0) \]
  and an injection $K_{0} \mathbb{D}_{+} \rTo^{\subseteq} K_{0} \Rp$.
  --- These statements hold \textit{mutatis mutandis} for the
  analogous maps $K_{q} \mathbb{D}_{-} \rTo K_{q} \Rn$ and
  $K_{q} \mathbb{D}_{0} \rTo K_{q} R$.
\end{lemma}

\begin{proof}
  As remarked before, the functor $\cY \mapsto Y^{+}$ establishes an
  equivalence of $\mathbb{C}_{+}$ with the category of bounded chain
  complexes of finitely generated projective $\Rp$\nbd-modules, and
  $\mathbb{D}_{+}$ corresponds to the subcategory of complexes~$Y^{+}$
  such that $Y^{+} \tensor_{\Rp} \R$ consists of modules which are
  stably induced from~$\Rn$. Let $\mathbb{D}'$ and $\mathbb{C}'$
  denote the full subcategories of complexes concentrated in chain
  level~$0$. Then the map of $K$\nbd-groups in question can be
  computed using \textsc{Quillen}'s $Q$\nbd-construction, applied to
  the inclusion of exact categories
  $\mathbb{D}' \subseteq \mathbb{C}'$. The result is now an immediate
  consequence of \textsc{Grayson} cofinality
  \cite[Theorem~1.1]{MR559850} since $\mathbb{D}'$ is closed under
  extension in $\mathbb{C}'$; indeed, an object $\cY \in \mathbb{C}'$
  is in~$\mathbb{D}'$ if and only if
  $\alpha \big( [Y^{0}] \big) = 0 \in K_{0} (\Rn \downarrow \R)$, and
  $K$\nbd-theory is additive on short exact sequences. To see that the
  former category is cofinal in the latter it suffices to observe that
  every finitely generated projective module can be complemented to a
  finitely generated free one, which is automatically stably induced.
\end{proof}

\subsection*{The corners of the fundamental square}

We can now identify the corners of the fundamental square as the
algebraic $K$\nbd-theory spaces of the categories $\mathbb{D}^{-}$,
$\mathbb{D}^{0}$ and~$\mathbb{D}^{+}$.

\begin{lemma}
  \label{lem:forget-C-D}
  The forgetful functor $\Phi^{+} \colon \cC \rTo
  \mathbb{D}_{+}$, given by
  \begin{displaymath}
    \cY = \big( Y^{-} \rTo Y^{0} \lTo Y^{+} \big) %
    \quad \mapsto \quad %
    \Phi^{+}(\cY) = \big( Y^{0} \lTo Y^{+} \big) \ ,%
  \end{displaymath}
  induces a homotopy equivalence on
  $\mathcal{S}_{\bullet}$-constructions with respect to
  $h_{+}$-equivalences:
  \begin{displaymath}
    h_{+} \sdot \cC \rTo^{\simeq} h_{+} \sdot \mathbb{D}_{+}
  \end{displaymath}
  The statement holds \textit{mutatis mutandis} for $\mathbb{D}_{0}$
  and~$\mathbb{D}_{-}$ as well.
\end{lemma}

\begin{proof}
  By definition of $h_{+}$\nbd-equivalences the forgetful functor
  respects and detects weak equivalences. Let
  \begin{displaymath}
    \cX = \big( X^{-} \rTo^{\xi^{-}} X^{0} \lTo^{\xi^{+}} X^{+} \big)
  \end{displaymath}
  be an object of~$\cC$, and let
  \begin{displaymath}
    \overleftarrow{\cY} = \big( Y^{0} \lTo^{\upsilon^{+}} Y^{+} \big)
  \end{displaymath}
  be an object of~$\mathbb{D}_{+}$.  Given a morphism
  \begin{displaymath}
    g = (g^{0}, g^{+}) \colon \Phi^{+}(\cX) \rTo \overleftarrow{\cY}
  \end{displaymath}
  we construct the object $\cZ \in \cC$ associated
  with~$\overleftarrow{\cY}$ as described in 
  Lemma~\ref{lem:D_stably_extends}; denote the composition of~$g$ with
  the inclusion $\overleftarrow{\cY} \rTo \Phi^{+}(\cZ)$ by
  $h = (h^{0}, h^{+})$. The composite map
  \begin{displaymath}
    X^{-} \rTo^{\xi^{-}} X^{0} \rTo^{h^{0}} Z^{0} 
  \end{displaymath}
  factors as
  \begin{displaymath}
    X^{-} \rTo^{h^{-}} Z^{-} \tensor_{\Rn} \Rnn{q} \rTo^{\subseteq}
    Z^{-} \tensor_{\Rn} \R \iso Z^{0}
  \end{displaymath}
  for sufficiently large $q \geq 0$; this is because
  $X^{-}$ is a bounded complex of finitely generated
  $\Rn$\nbd-modules, and because
  \[Z^{0} \iso Z^{-} \tensor_{\Rn} \R = Z^{-} \tensor_{\Rn} \bigcup_{k
    \geq 0} \Rnn{k} = \bigcup_{k \geq 0} Z^{-} \tensor_{\Rn} \Rd{k} \
  .\] This results in a map
  \begin{displaymath}
    h = (h^{-}, h^{0}, h^{+}) \colon \cX \rTo \cZ(q,0)
  \end{displaymath}
  in~$\cC$ (we identify the canonically isomorphic complexes
  $\Phi^{+}(\cZ) = Z^{+}$ and
  $\Phi^{+} \big( \cZ(q,0) \big) = Z^{+} \tensor_{\Rp} \R$ here). The
  projection map
  $p \colon \Phi^{+}\big(\cZ(q,0)\big) \rTo \overleftarrow{\cY}$ is an
  $h_{+}$\nbd-equivalence, and satisfies the condition
  $p \circ \Phi^{+}(h) = g$. By \textsc{Waldhausen}'s approximation
  theorem \cite[Theorem~1.6.7]{MR802796}, $\Phi^{+}$ induces a
  homotopy equivalence on $\sdot$\nbd-con\-struc\-tions.
\end{proof}

\section{Establishing the Mayer-Vietoris sequence}
\label{sec:MV}

We will now establish the \textsc{Mayer}-\textsc{Vietoris} sequence of
Theorem~\ref{thm:MV-sequence}. As before let $\R$ be a strongly
$\bZ$-graded ring. The induction functors
\begin{displaymath}
  j^{-}_{*} = (\nix \tensor_{\Rn} R) \qquad \text{and} \qquad %
  j^{+}_{*} = (\nix \tensor_{\Rp} R)
\end{displaymath}
give rise to maps
$\gamma = j^{-}_{*} - j^{+}_{*} \colon K_{q} (\Rn) \oplus K_{q} (\Rp)
\rTo K_{q} (R)$, for $q \geq 0$.

In view of Lemma~\ref{lem:forget-C-D}, the homotopy cartesian
fundamental square of the projective line~(\ref{eq:fund}) yields a
\textsc{Mayer}-\textsc{Vietoris} sequence of algebraic $K$-groups
\begin{multline*}
  \ldots \rTo^{\gamma} K_{q+1} \mathbb{D}_{0} %
  \rTo^{\delta} K_{q} \mathbb{P}^1 %
  \rTo[l>=6em]^{\beta = (\beta_{-}, \beta_{+})} K_{q} \mathbb{D}_{-}
  \oplus K_{q} \mathbb{D}_{+} \\%
  \rTo^{\gamma = \gamma_{-} - \gamma_{+}} K_{q} \mathbb{D}_{0} %
  \rTo^{\delta} \ldots \ ,
\end{multline*}
for $q \geq 0$, ending with a surjective homomorphism
$K_{0} \mathbb{D}_{-} \oplus K_{0} \mathbb{D}_{+} \rTo K_{0}
\mathbb{D}_{0}$.
By Lemma~\ref{lem:correct_K0} this sequence coincides, for $q>0$, with
the one of Theorem~\ref{thm:MV-sequence}, so we only need to look at
its tail end. Using Lemma~\ref{lem:correct_K0} again we construct the
following commutative diagram:
\begin{diagram}
  K_{1} (R) & \rTo & K_{0} (\pp) %
  & \rTo & K_{0} (\Rn) \oplus K_{0} (\Rp) %
  & \rTo^{\bar \gamma} & K_{0} (R) %
  & \rTo & \kay (\Rd{0}) & \rTo 0\\%
  \uTo<{=} && \uTo<{=} && \uTo<{\subseteq} && \uTo<{\subseteq} \\
  K_{1} \mathbb{D}_{0} & \rTo & K_{0} (\pp) %
  & \rTo^{\beta} & K_{0} \mathbb{D}_{-} \oplus K_{0} \mathbb{D}_{+} %
  & \rTo^{\gamma} & K_{0} \mathbb{D}_{0} %
  & \rTo & 0 \\%
\end{diagram}
We know that the bottom row is exact, and argue that this remains true
for the top row. --- Replacing the target of~$\beta$ by a larger group
does not change $\ker(\beta)$, so the top row is exact at
$K_{0}(\pp)$. It is exact at $K_{0} (R)$ and~$\kay (\Rd{0})$ by
definition of the latter group. Thus it remains to verify exactness
at the third entry.

So let $(x,y) \in \ker \bar \gamma$ be given. We can find finitely
generated projective modules $P$ and $Q$ over $\Rn$ and~$\Rp$,
respectively, and numbers $p, q \geq 0$, such that
\begin{displaymath}
  x = [P] - [\Rn^{p}] \in K_{0} (\Rn) \qquad \text{and} \qquad y = [Q]
  - [\Rp^{q}] \in K_{0} (\Rp) \ .
\end{displaymath}
 The condition $\bar\gamma (x,y) = 0$ translates into the equality
 \begin{displaymath}
   [P \tensor_{\Rn} R] - [\R^{p}] = [Q \tensor_{\Rp} R] -
   [\R^{q}]
 \end{displaymath}
in~$K_{0} (R)$. Consequently, there exists $s \geq 0$ such that
\begin{displaymath}
  \big( P \tensor_{\Rn} R \big) \oplus \R^{q} \oplus \R^{s} %
  \iso %
  \big( Q \tensor_{\Rp} R \big) \oplus \R^{p} \oplus \R^{s} \ .
\end{displaymath}
As the right-hand module is induced from~$\Rp$ this shows that
$P \tensor_{\Rn} R$ is stably induced from~$\Rp$ so that
the diagram
\begin{displaymath}
  P \rTo P \tensor_{\Rn} R
\end{displaymath}
defines an object of~$\mathbb{D}_{-}$. Since $\Rn^{p} \rTo \R^{p}$ is
an object of~$\mathbb{D}_{-}$ as well we conclude that
$x \in K_{0} \mathbb{D}_{-}$. By a symmetric argument we can show
$y \in K_{0} \mathbb{D}_{+}$. As $\gamma(x,y) = 0$ (the fourth
vertical map is injective), and as the bottom row is exact, we know
that $(x,y) = \beta(z)$ for some $z \in K_{0}(\pp)$, which shows
exactness of the top row. This completes the proof of
Theorem~\ref{thm:MV-sequence}.\qed

\section{Proof of the fundamental theorem}

\subsection*{The $K$-theory of the projective line revisited}

Recall from Theorem~\ref{thm:HM} that there are
isomorphisms of $K$\nbd-groups
\begin{displaymath}
  \tilde \alpha = \big( \Psi_{-1,0} \ \, \Psi_{0,0} \big) \colon %
  K_{q} \Rd{0} \oplus K_{q} \Rd{0} \rTo K_{q} \pp \ , \ %
  \big( [P], [Q] \big) \mapsto \big[ \Psi_{-1,0} (P) \big] %
  + \big[ \Psi_{0,0} (Q) \big] \ ,
\end{displaymath} 
where $\Psi_{k,\ell}$ stands for the canonical sheaf functors
(Definition~\ref{def:psi}). By pre-com\-pos\-ing $\tilde \alpha$ with the
invertible map
$\left( \begin{smallmatrix} -1&0\\1&1 \end{smallmatrix} \right)$, we
obtain modified isomorphisms
\begin{equation}
  \label{eq:modified_iso}
  \begin{multlined}
    \alpha = \big( -\Psi_{-1,0} + \Psi_{0,0} \quad \Psi_{0,0} \big)
    \colon %
    K_{q} \Rd{0} \oplus K_{q} \Rd{0} \rTo K_{q} \pp \ , \qquad
    \\ %
    \qquad
    \big( [P], [Q] \big) \mapsto -\big[ \Psi_{-1,0} (P) \big] + %
    \big[ \Psi_{0,0} (P) \big] %
    + \big[ \Psi_{0,0} (Q) \big] \ .
  \end{multlined}
\end{equation}
(This map cannot be confused with the map~$\alpha$
from~\eqref{eq:alpha}.) As we have a cylinder functor at our
disposal, the additivity theorem implies that we can model the minus
sign by taking suitable ``homotopy cofibres'' of maps of
functors. Concretely, the isomorphisms~$\alpha$ are induced by the
functor
\begin{equation}
  \begin{multlined}
    A \colon \chp \times \chp \rTo \vb_{0} \ , \\ %
    (C,D) \mapsto \cone \big( \Psi_{-1,0} (C) \rTo \Psi_{0,0} (C)
    \big) %
    \oplus \Psi_{0,0} (D) \qquad\qquad\qquad
  \end{multlined}\label{eq:functor_A}
\end{equation}
and the ensuing homotopy equivalence of $K$\nbd-theory spaces
\begin{displaymath}
   \cone ( \Psi_{-1,0} \rTo \Psi_{0,0}) + \Psi_{0,0} \colon K(\Rd{0})
     \times K(\Rd{0}) \rTo K(\pp) \ ,
\end{displaymath}
where ``$+$'' refers to the $H$-space structure given by direct sum.
Alternatively, we can use the functor
\begin{multline*}
  \qquad\qquad\qquad A' \colon \chp \times \chp \rTo \vb_{0} \ , \\ %
  (C,D) \mapsto \Sigma \Psi_{-1,0} (C) \,\oplus\, \Psi_{0,0} (C)
  \,\oplus\, \Psi_{0,0} (D) \ ;
  \qquad\qquad\qquad
\end{multline*}
the maps induced by~$A$ and~$A'$ are homotopic, by the additivity
theorem.

\medbreak

It will be convenient to have an explicit homotopy inverse
for~$\alpha$. We record the following fact:

\begin{lemma}
  \label{lem:homotopy_inverse_alpha}
  The functor
  \begin{displaymath}
    \Xi \colon \vb_{0} \rTo \chp \times \chp \ , \ \cY \mapsto
    \begin{pmatrix}
      \Sigma \Gamma \cY(1,0) \,\oplus\, \Gamma \cY \,\oplus\, \Gamma
      \cY(1,-1) \\
      \Sigma \Gamma \cY(1,-1) \,\oplus\, \Gamma \cY(1,0)
    \end{pmatrix}
  \end{displaymath}
  induces a homotopy inverse on the level of $K$\nbd-theory spaces;
  here $\Gamma$ is the ``global sections'' functor from
  Definition~\ref{def:cohomology}.
\end{lemma}

\begin{proof}
  The functor~$A$ is known to induce a homotopy equivalence. Hence in
  view of the above remarks on~$A'$ it is enough to show that the map
  induced by $\Xi \circ A'$ is homotopic to the identity. Recalling
  the natural isomorphisms
  \begin{displaymath}
    \Gamma \Psi_{k,\ell} P \iso \bigoplus_{j=-\ell}^{k} s_{j} P
    \tensor_{\Rd{0}} \Rd{j} \qquad \text{and} \qquad \big( \Psi_{k,\ell}
    P \big)(a,b) \iso \Psi_{k+a, \ell+b} P \ ,
  \end{displaymath}
  from Proposition~\ref{prop:Gamma_Psi}, where
  $s_{j} P = P \tensor_{\Rd{0}} \Rd{j}$ is the $j$th shift of~$P$, one
  calculates readily that the first component of $\Xi \circ A' (C,D)$
  is
  \begin{align*}
    \ & \ 
        \Sigma^2 \Psi_{0,0} (C) \,\oplus\,
        \Sigma \Gamma \Psi_{1,0}(C)   \,\oplus\, \Sigma \Gamma
        \Psi_{1,0} (D) \\ %
      & \ \,\oplus \quad %
        \Sigma \Gamma \Psi_{-1,0}(C) \,\oplus\, \Gamma \Psi_{0,0} (C)
        \,\oplus\, \Gamma Y_{0,0} (D) \\ %
      & \ \,\oplus \quad %
        \Sigma \Gamma \Psi_{0,-1} (C) \,\oplus\, \Gamma \Psi_{1,-1} (C)
        \,\oplus\, \Gamma \Psi_{1,-1} (D) \\
    \iso \quad & \ \Sigma^2 C \,\oplus \, \Sigma C \,\oplus \, \Sigma s_{1} C
                 \,\oplus \, \Sigma D \,\oplus \, \Sigma s_{1} D \\
      & \ \,\oplus \quad C \,\oplus\, D \\
      & \ \,\oplus \quad s_{1} C \,\oplus\, s_{1} D \ ,
  \end{align*}
  while the second component is
  \begin{align*}
    \ & \ 
        \Sigma^2 \Psi_{0,-1} (C) \,\oplus\,
        \Sigma \Gamma \Psi_{1,-1}(C)   \,\oplus\, \Sigma \Gamma
        \Psi_{1,-1} (D) \\ %
      & \ \,\oplus \quad %
        \Sigma \Gamma \Psi_{0,0}(C) \,\oplus\, \Gamma \Psi_{1,0} (C)
        \,\oplus\, \Gamma Y_{1,0} (D) \\ %
    \iso \quad & \ \Sigma s_{1} C
                 \,\oplus \, \Sigma s_{1} D \\
      & \ \,\oplus \quad \Sigma C \,\oplus\, C \,\oplus\, s_{1} C
        \,\oplus\, D \,\oplus\, s_{1} D \ .
  \end{align*}
  Recalling that suspension represents the $H$\nbd-space structure
  inverse on $K$\nbd-theory spaces, this shows that the
  composition~$\Xi \circ A'$ induces a map that is homotopic to the
  identity as required.
\end{proof}

\subsection*{The modified Mayer-Vietoris sequence}

On a much more elementary level, we have automorphisms
\begin{displaymath}
  \eta =
  \begin{pmatrix}
    \id & -i^{-}_{*} p^{+}_{*} \\
    0 & \id
  \end{pmatrix}
  \colon %
  K_{q} (\Rn) \oplus K_{q} (\Rp) \rTo K_{q} (\Rn) \oplus K_{q} (\Rp) %
  \qquad (q \geq 0)
\end{displaymath}
with inverse given by
\begin{displaymath}
  \eta\inv =
  \begin{pmatrix}
    \id & i^{-}_{*} p^{+}_{*} \\
    0 & \id
  \end{pmatrix}
  \colon %
  K_{q} (\Rn) \oplus K_{q} (\Rp) \rTo K_{q} (\Rn) \oplus K_{q} (\Rp) %
  \ .
\end{displaymath}
Both maps are induced by functors, with the minus sign modelled by
suspension (making implicit use of the additivity theorem and the
presence of a cylinder functor again). Explicitly, $\eta$ is induced
by the functor
\begin{multline*}
  \qquad\qquad
  \chpr \Rn \times \chpr \Rp \rTo \chpr \Rn \times \chpr \Rp \ , \\
  (C,D) \mapsto \big( C \oplus \Sigma i_{*}^{-} p_{*}^{+} D,\, D \big)
  \ . \qquad
\end{multline*}
We use the map~$\eta$ and its inverse to construct a modified
\textsc{Mayer}-\textsc{Vietoris} sequence
\begin{equation}
  \label{eq:modified-MV}
  \begin{split}
    \ldots \rTo^{\gamma \eta\inv} K_{q+1}R %
    \rTo^{\alpha\inv\delta} K_{q} \Rd{0} \oplus K_{q} \Rd{0} \qquad \\ %
    \qquad
    \rTo^{\eta\beta\alpha} K_{q} \Rn \oplus K_{q} \Rp %
    \rTo^{\gamma \eta\inv} K_{q} R %
    \rTo^{\alpha\inv\delta} \ldots \ \ , 
  \end{split}
\end{equation}
ending with the exact sequence
\begin{displaymath}
  K_{0} \Rd{0} \oplus K_{0} \Rd{0} \\ %
  \rTo^{\eta\beta\alpha} K_{0} \Rn \oplus K_{0} \Rp %
  \rTo^{\gamma \eta\inv} K_{0} R %
  \rTo \kay \Rd{0} \rTo 0 \ .
\end{displaymath}

Note that any exact sequence of the form $A \rTo^{f} B \rTo^{g} C$
gives rise to a short exact sequence
\begin{displaymath}
  0 \rTo A / \ker(f) \rTo B \rTo \im(g) \rTo 0 \ ;
\end{displaymath}
in this way, sequence~\eqref{eq:modified-MV} can be split into short
exact sequences
\begin{gather*}
  0 \rTo %
  \big( K_{q} \Rn \oplus K_{q} \Rp) %
  \big/ \ker(\gamma \eta\inv) %
  \rTo^{\gamma \eta\inv} K_{q} R %
  \rTo^{\alpha\inv \delta} \im (\alpha\inv \delta) \rTo 0
  \quad (q > 0) \\
  \intertext{and} %
  0 \rTo %
  \big( K_{0} \Rn \oplus K_{0} \Rp) %
  \big/ \ker(\gamma \eta\inv) %
  \rTo^{\gamma \eta\inv} K_{0} R %
  \rTo \kay \Rd{0} \rTo 0 \ .
\end{gather*}
From exactness of the modified \textsc{Mayer}-\textsc{Vietoris}
sequence~\eqref{eq:modified-MV} again we have the equalities
$\ker(\gamma \eta\inv) = \im (\eta\beta\alpha)$ and
$\im (\alpha\inv \delta) = \ker(\eta\beta\alpha)$, so we obtain
short exact sequences
\begin{gather}
  \label{eq:SES1}
  0 \rTo \mathrm{coker} (\eta\beta\alpha)
  \rTo K_{q} R \rTo \ker(\eta\beta\alpha) \rTo 0  \quad (q>0) \\
  \intertext{and} %
  \label{eq:SES2}
  0 \rTo \mathrm{coker} (\eta\beta\alpha)
  \rTo K_{0} R \rTo \kay \Rd{0} \rTo 0 \ .
\end{gather}

\subsection*{The map $\eta\beta\alpha$}

On the level of categories the effect of the map $\beta\alpha$ is
to send $(P,Q) \in \chp \times \chp$ to $(C,D) \in \chpr \Rn \times
\chpr \Rp$ where
\begin{gather*}
  \begin{multlined}
    C = \cone \big( P \tensor_{\Rd{0}} \Rnn {-1} \rTo^{\subseteq} P
    \tensor_{\Rd{0}} \Rn \big) \ \oplus \ \big( Q \tensor_{\Rd{0}} \Rn
    \big) \\ %
    \qquad = i^{-}_{*} \cone (s_{-1} P \rTo P) \oplus i^{-}_{*} Q
  \end{multlined}
  \\
  \intertext{and} %
  \begin{multlined}
    D = \cone \big( P \tensor_{\Rd{0}} \Rp \rTo^{=} P \tensor_{\Rd{0}}
    \Rp \big) \ \oplus \ \big( Q \tensor_{\Rd{0}} \Rp \big) \qquad \\ \simeq
    \ %
    Q \tensor_{\Rd{0}} \Rp \ = \ i^{+}_{*} Q \ ;
  \end{multlined}
\end{gather*}
here $s_{-1}$ denotes the shift functor
\begin{displaymath}
  s_{-1} \colon Q \mapsto Q \tensor_{\Rd{0}} \Rd{-1}
\end{displaymath}
which maps the category of finitely generated projective
$\Rd{0}$\nbd-modules to itself. --- To determine the effect of the
map~$\eta$ on~$(C,D)$, recall first that weakly equivalent functors
induce homotopic maps on $K$-theory spaces. So we can, for example,
replace $D$ by~$i^{+}_{*}Q$. The first component of
$\eta\beta\alpha(P,Q)$ then is, up to homotopy,
\begin{displaymath}
  i^{-}_{*} \cone (s_{-1} P \rTo P) \,\oplus\, i_{*}^{-} Q \,\oplus\,
  \Sigma i^{-}_{*} p^{+}_{*} i^{+}_{*} Q \ .
\end{displaymath}
Since $p^{+}_{*} i^{+}_{*} Q = Q$, and since suspension represents a
homotopy inverse for direct sum \cite[Proposition~1.6.2]{MR802796}, we
can thus model the effect of $\eta\beta\alpha$ by the functor
\begin{displaymath}
  (P,Q) \mapsto \big( i^{-}_{*} \cone (s_{-1} P \rTo P) , \ %
  i^{+}_{*} Q \big) \ .
\end{displaymath}
On the level of $K$\nbd-groups, this means that the effect
of~$\eta\beta\alpha$ is described by the diagonal matrix
\begin{equation}
  \label{eq:diagonal}
  \eta\beta\alpha =
  \begin{pmatrix}
    i^{-}_{*} (\id - s_{-1}) & 0 \\
    0 & i^{+}_{*}
  \end{pmatrix}
  \colon K_{q} \Rd{0} \oplus K_{q} \Rd{0} \rTo K_{q} \Rn \oplus
  K_{q} \Rp \ .
\end{equation}

\subsection*{The kernel of $\eta\beta\alpha$}

In view of~\eqref{eq:diagonal},
$\ker(\eta\beta\alpha) = \ker \big( i^{-}_{*} (s_{-1} - \id)
\big)\oplus \ker (i^{-}_{*})$.
Now the maps $i^{\mp}_{*}$ are split injective (with left inverse
$p^{\mp}_{*}$), thus
\begin{equation}
  \label{eq:identify_ker}
  \ker(\eta\beta\alpha) \iso \ker(\sd_{*}) \oplus \{0\} %
  = \sker+q \Rd{0} \ ,
\end{equation}
where $\sd_{*}$ denotes the shift difference map
$\sd_{*} = \id - s_{-1}  \colon K_{q} \Rd{0} \rTo K_{q} \Rd{0}$
from~\eqref{eq:def_sd} with kernel $\sker+q \Rd{0}$ 
(\hglue0.7pt Definition~\ref{def:shift_kernel}).

\subsection*{The cokernel of $\eta\beta\alpha$}

From the specific representation of~$\eta\beta\alpha$
in~\eqref{eq:diagonal} we read off that
\begin{displaymath}
  \coker(\eta\beta\alpha) = \coker (i^{-}_{*} \circ \sd_{*}) \oplus
  \coker (i^{+}_{*}) \ ,
\end{displaymath}
the second summand being nothing but $\NK+q \Rd{0}$ by definition. To
identify the first summand we compose $i^{-}_{*} \circ \sd_{*}$ with
the splitting isomorphism~\eqref{eq:split-nil}:
\begin{displaymath}
  K_{q} (\Rd{0}) \rTo^{\sd_{*}} K_{q} (\Rd{0}) \rTo^{i^{-}_{*}} K_{q}
  (\Rn) \rTo[l>=4em]^{S = (c, p^{-}_{*})}_{\iso} \NK-q \oplus K_{q}
  (\Rd{0})
\end{displaymath}
(Here $c \colon K_{q} (\Rn) \rTo \NK-q$ is the canonical projection
onto the cokernel of~$i^{-}_{*}$.) Thus
$\coker (i^{-}_{*} \circ \sd_{*}) \iso \coker (S \circ i^{-}_{*} \circ
\sd_{*})$.
Since $c \circ i^{-}_{*} = 0$, and since $p^{-}_{*} \circ i^{-}_{*}$
is the identity, we conclude that this group is isomorphic to the
cokernel of the composition
\begin{displaymath}
  K_{q} \Rd{0} \rTo^{\sd_{*}} %
  K_{q} \Rd{0} \rTo^{\subseteq} %
  \NK-q \Rd{0} %
  \oplus %
  K_{q} \Rd{0} \ ,
\end{displaymath}
which is $\NK-q \oplus \scoker-q \Rd{-}$. In total, this results in an
isomorphism
\begin{equation}
  \label{eq:identify_coker}
  \coker(\eta\beta\alpha) \iso \NK-q \Rd{0} \oplus \scoker-q \Rd{0}
  \oplus \NK+q \Rd{0} \ .
\end{equation}

\medbreak

Putting the identifications~\eqref{eq:identify_ker}
and~\eqref{eq:identify_coker} of the kernel and cokernel of
$\eta\beta\alpha$ into the sequences \eqref{eq:SES1}
and~\eqref{eq:SES2} gives precisely the advertised Fundamental
Theorem~\ref{thm:fundamental}. \qed

\section{The localisation sequence}
\label{sec:loc_seq}

In this section we will establish the long exact ``localisation''
sequence of Theorem~\ref{thm:loc-and-nil}. There are in fact two
``mirror-symmetric'' versions, which we state in full for
completeness.

\begin{theorem}[``Localisation sequence'']
  \label{thm:localisation-sequence}
  Let $\R$ be a strongly $\bZ$\nbd-graded ring. There are long exact
  sequences of algebraic $K$\nbd-groups \addtocounter{equation}{-1}
  \begin{subequations}
    \begin{gather}
      \begin{multlined}
        \ldots \rTo K_{q+1} \R \rTo \nil+q (\Rd{0}) \rTo^{\phi} K_{q}
        \Rp
        \rTo K_{q} R \\[0.5em]
        \rTo \ldots \rTo \nil+0 (\Rd{0}) \rTo^{\phi} K_{0} \Rp \rTo
        K_{0} \R
      \end{multlined}
      \label{eq:loc-seq-plus}
      \\
      \intertext{and} %
      \begin{multlined}
        \ldots \rTo K_{q+1} \R \rTo \nil-q (\Rd{0}) \rTo^{\phi} K_{q}
        \Rn
        \rTo K_{q} R \\[0.5em]
        \rTo \ldots \rTo \nil-0 (\Rd{0}) \rTo^{\phi} K_{0} \Rn \rTo
        K_{0} \R
      \end{multlined}
      \label{eq:loc-seq-minus}
    \end{gather}
  \end{subequations}
  with $\phi$ induced by the forgetful functor $(Z, \zeta) \mapsto Z$
  on the category $\cT{\pm}$. The groups $\nil\pm{q}(\Rd{0})$ are as
  defined in~\eqref{eq:notation_nil+} and~\eqref{eq:notation_nil-}.
\end{theorem}

\begin{proof}
  We will show that the sequence~\eqref{eq:loc-seq-plus} is exact;
  the argument for sequence~\eqref{eq:loc-seq-minus} is similar.

  We apply \textsc{Waldhausen}'s fibration theorem to the right-hand
  vertical map in the fundamental square~\eqref{eq:fund}. In view of
  Lemma~\ref{lem:forget-C-D} this results in exact sequences of
  $K$-groups
  \begin{displaymath}
    \ldots \rTo K_{q+1} \mathbb{D}_{+} \rTo K_{q+1} \mathbb{D}_{0}
    \rTo %
    \pi_{q} \Omega | h_{+} \vb^{h_{0}}|
    \rTo^{k'} K_{q} \mathbb{D}_{+} \rTo K_{q}
    \mathbb{D}_{0}
  \end{displaymath}
  for $q \geq 0$. In view of Lemma~\ref{lem:correct_K0} and
  Theorem~\ref{thm:fibre-terms-are-nil}, this proves exactness of the
  sequence~\eqref{eq:loc-seq-plus} in the $q \geq 1$ range down to the
  term $K_1 \Rp$. Lemma ~\ref{lem:correct_K0} also states that in the
  commutative diagram in Fig.~\ref{fig:diag_localisation}
  \begin{figure}[ht]
    \centering
    \begin{diagram}
      K_{1} \Rp & \rTo & K_{1} \R & \rTo & 
      \pi_{0} \Omega | h_{+} \vb^{h_{0}}| & \rTo^{k} &
      K_{0} \Rp & \rTo^{j^{+}_{*}} & K_{0} \R \\
      \uTo<{=} && \uTo<{=} && \uTo<{=} && \uTo<{\subseteq}>{f} &&
      \uTo<{\subseteq}>{g} \\
      K_{1} \mathbb{D}_{+} & \rTo & K_{1} \mathbb{D}_{0} & \rTo &
      \pi_{0} \Omega | h_{+} \vb^{h_{0}}| & \rTo^{k'} & K_{0}
      \mathbb{D}_{+} & \rTo^{j'} & K_{0}
      \mathbb{D}_{0} \\
    \end{diagram}
    \caption{Diagram used to establish the tail end of the localisation sequence}
    \label{fig:diag_localisation}
  \end{figure}
  the two vertical maps labelled $f$ and~$g$ are injective. As the
  bottom row is exact, the top row is automatically exact as well
  except possibly at~$K_{0} \Rp$. Let $x \in K_{0} \Rp$ be an element
  with $j^{+}_{*}(x) = 0$. We can write $x = [P] - [Q]$, for finitely
  generated projective $\Rp$\nbd-modules $P$ and~$Q$.  The condition
  $j^{+}_{*} [P] - j^{+}_{*} [Q] = j^{+}_{*}(x) = 0 \in K_{0} \R$,
  that is, $[P \tensor_{\Rp} \R] = [Q \tensor_{\Rp} \R] \in K_{0} \R$,
  means that there exists a number $\ell \geq 0$ together with an
  isomorphism
  \begin{displaymath}
    j^{+}_{*}(P) \oplus \R^{\ell} = \big( P \tensor_{\Rp} \R \big)
    \,\oplus\, \R^{\ell} \iso \big( Q \tensor_{\Rp} \R \big) \,\oplus\,
    \R^{\ell} = j^{+}_{*}(Q) \oplus \R^{\ell} \ .
  \end{displaymath}
  Let $P'$ be a complement of~$P$ so that $P' \oplus P$ is a finitely
  generated free $\Rp$\nbd-module. Then $j^{+}_{*}(P' \oplus P) \oplus
  \R^{\ell}$ is a finitely generated free $\R$\nbd-module, and
  \begin{displaymath}
    j^{+}_{*}(P' \oplus P) \oplus \R^{\ell} = j^{+}_{*}(P') \oplus
    j^{+}_{*}(P) \oplus \R^{\ell} \iso j^{+}_{*}(P') \oplus
    j^{+}_{*}(Q) \oplus \R^{\ell} = j^{+}_{*}(P' \oplus Q) \oplus \R^{\ell}
  \end{displaymath}
  so that $j^{+}_{*}(P' \oplus Q) \oplus \R^{\ell}$ is a finitely
  generated free $\R$ module as well. Consequently, both
  $j^{+}_{*}(P' \oplus P)$ and $j^{+}_{*} (P' \oplus Q)$ stably extend
  to~$\Rn$; thus the two diagrams
  \begin{displaymath}
    p = \Big( j^{+}_{*}(P' \oplus P) \lTo P' \oplus P \Big) %
    \qquad \text{and} \qquad
    q = \Big( j^{+}_{*}(P' \oplus Q) \lTo P' \oplus Q \Big) %
  \end{displaymath}
  are objects of~$\mathbb{D}_{+}$. Let us consider the element
  $z = [p]-[q] \in K_{0} \mathbb{D}_{+}$. We calculate
  \begin{displaymath}
    f(z) = f \big( [p]-[q] \big) = [P' \oplus P] - [P' \oplus Q] =
    [P]-[Q] = x
  \end{displaymath}
  (see Lemma~\ref{lem:correct_K0} for the effect of~$f$). As
  $x \in \ker(j^{+}_{*})$, and as $g$ is an injection, this implies
  $z \in \ker(j')$, and by exactness of the lower horizontal sequence
  we infer that $z = k'(y)$ for some
  $y \in \pi_{0} \Omega | h_{+} \vb^{h_{0}}|$. Then
  $k(y) = fk'(y) = f(z)=x$ so that $x \in \im k$. This proves the top
  row to be exact at $K_{0} \Rp$, and establishes together with
  Theorem~\ref{thm:fibre-terms-are-nil} the tail end of the long exact
  sequence.

  \medbreak

  \newcommand{\bfdr}{\overline{\mathbf{FD}}}

  It remains to identify the map~$\phi$ in the
  sequence~\eqref{eq:loc-seq-plus}. To this end, let $\bfdr(\Rp)$
  denote the category of bounded complexes of $\Rp$\nbd-modules
  which are quasi-iso\-mor\-phic to a bounded complex of finitely
  generated projective $\Rp$\nbd-mod\-ules, and consist of projective
  $\Rd{0}$\nbd-mod\-ules. The inclusion functor
  \[\chpr{\Rp} \rTo \bfdr(\Rp)\] is exact and yields isomorphisms on algebraic
  $K$\nbd-groups (with respect to weak equivalences the
  quasi-isomorphisms, and cofibrations the mono\-mor\-phisms with
  levelwise $\Rd{0}$-projective cokernel) as can be checked with the
  help of \textsc{Waldhausen}'s approximation theorem. This inclusion
  is the right hand vertical map in the square diagram of
  Fig.~\ref{fig:loc-seq}.
  \begin{figure}[ht]
    \centering
    \begin{diagram}
      \ch \vb^{h_{0}} & \rTo & \chpr{\Rp} \\
      \dTo && \dTo<{\subseteq} \\
      \cT{+} & \rDashto[l>=3em] & \bfdr(\Rp)
    \end{diagram}
    \caption{Diagram used in proof of localisation sequence}
    \label{fig:loc-seq}
  \end{figure}
  The top horizontal arrow maps the complex of vector bundles~$\cZ$ to
  its component~$Z^{+}$. The left-hand vertical map is the one
  discussed in Theorem~\ref{thm:fibre-terms-are-nil}, identifying the
  group $\pi_{q} \Omega | h_{+} \vb^{h_{0}}|$ with $\nil+q(\Rd{0})$;
  it is given by the assignment
  \begin{displaymath}
    \cZ = \Big( Z^{-} \rTo Z^{0} \lTo^{\zeta^{+}} Z^{+} \Big) \
    \mapsto \ (Z^{+}, \zeta)
  \end{displaymath}
  with $\zeta$ the composition
  $Z^{+} \tensor_{\Rd{0}} \Rd{1} \iso Z^{+} \tensor_{\Rp} \Rpn{1} \rTo
  Z^{+} \tensor_{\Rp} \Rp \iso Z^{+} \tensor_{\Rd{0}} \Rd{0}$.
  Note that $Z^{+}$ consists of projective $\Rd{0}$\nbd-modules since
  $\R$ is \textit{strongly} $\bZ$\nbd-graded, that $Z^{+}$ is
  $\Rd{0}$\nbd-finitely dominated by Lemma~\ref{lem:R-acyclyc-R0-fd},
  and that $\zeta$ is homotopy nilpotent by
  Corollary~\ref{cor:nilpotent_iff_acyclic}. Using the functor
  $(Z, \zeta) \mapsto Z^{\zeta}$ in the lower horizontal position
  renders the square diagram commutative; the complex $Z^{\zeta}$ is
  indeed an object of~$\bfdr(\Rp)$ by Remark~\ref{rem:nil-cat}. The
  induced map on algebraic $K$\nbd-groups is the map~$\phi$ occurring
  in the sequence~\eqref{eq:loc-seq-plus}. --- The proof is complete.
\end{proof}

\section{The $K$-theory of homotopy nilpotent twisted endomorphisms}

The nil groups $\NK\mp 0 \Rd{0}$ can be interpreted as obstruction
groups for finitely generated projective modules over $\Rn$ or~$\Rp$,
respectively, to be stably induced from~$\Rd{0}$
(Proposition~\ref{prop:stably_induced_module}). For $q>0$ the groups
$\NK\pm q \Rd{0}$ are isomorphic to the (reduced) algebraic
$K$\nbd-groups of the category of nilpotent twisted endomorphisms, as
will be shown in this section. This will also complete the proof of
Theorem~\ref{thm:loc-and-nil}.

\begin{theorem}[Nil groups and $K$-theory of nilpotent endomorphisms]
  \label{thm:NK_is_nil}
  For $q>0$ there are natural isomorphisms of algebraic $K$\nbd-groups
  \addtocounter{equation}{-1}
  \begin{subequations}
    \begin{gather}
      \label{eq:nk+} %
      \NK+q (\Rd{0}) \rTo^{\iso} \ker \Big( o^{-}_{*} \colon
      \nil-{q-1} \Rd{0} \rTo %
      K_{q-1} \Rd{0} \Big) \\
      \intertext{and} %
      \label{eq:nk-} %
      \NK-q (\Rd{0}) \rTo^{\iso} \ker \Big( o^{+}_{*} \colon
      \nil+{q-1} \Rd{0} \rTo %
      K_{q-1} \Rd{0} \Big) \ ,
    \end{gather}
  \end{subequations}
  with the groups $\nil\mp {q-1} \Rd{0}$ from~\eqref{eq:notation_nil+}
  and~\eqref{eq:notation_nil-}, and maps $o^{\mp}_{*}$ as introduced
  in~\eqref{eq:maps-o} induced by the forgetful functors
  $o^{\mp} \colon (Z,\zeta) \mapsto Z$.
\end{theorem}

\begin{proof}
  Upon taking the homotopy fibre of the top horizontal map in the
  fundamental square~\eqref{eq:fund} we obtain a long exact sequence
  of $K$-groups containing the snippet
  \begin{displaymath}
    K_{q} \pp \rTo K_{q} \mathbb{D}_{+} \rTo^{\nabla} K_{q-1} \vb^{h_{+}}
    \rTo^{\iota} K_{q-1} \pp \rTo K_{q-1} \mathbb{D}_{+} 
    \ ,
  \end{displaymath}
  with $\iota$ induced by the inclusion functor
  $\ch \vb^{h_{+}} \rTo \ch \vb$; the symbol~$\nabla$ stands for a
  connecting homomorphism in the long exact sequence of homotopy
  groups associated with the fibration. --- Using the isomorphism
  $\alpha \colon K_{q} \Rd{0} \oplus K_{q} \Rd{0} \rTo K_{q} \pp$
  from~\eqref{eq:modified_iso}, and with $K_{k} (\Rp)$ in place
  of~$K_{k} \mathbb{D}_{+}$ (Lemma~\ref{lem:correct_K0}), we obtain
  the exact sequence
  \begin{multline*}
    \quad\qquad %
    K_{q} \Rd{0} \oplus K_{q} \Rd{0} \rTo^{\epsilon} K_{q} \Rn %
    \rTo^{\nabla} K_{q-1} \vb^{h_{+}} \\ %
    \rTo^{\alpha\inv\iota} K_{q-1} \Rd{0} \oplus K_{q-1} \Rd{0} %
    \rTo^{\epsilon} K_{q-1} \Rn \ . %
    \quad\qquad %
  \end{multline*}
  The map $\epsilon$ is the difference of maps induced by the functors
  \begin{displaymath}
    (P,Q) \mapsto \big( P \tensor_{\Rd{0}} \Rp \big) \,\oplus\, %
    \big( Q \tensor_{\Rd{0}} \Rp \big) %
    \qquad \text{and} \qquad %
    (P,Q) \mapsto P \tensor_{\Rd{0}} \Rp \ ,
  \end{displaymath}
  thus $\epsilon$ is the usual split injection induced from
  $i^{+} \colon Q \mapsto Q \tensor_{\Rd{0}} \Rp$ on the second
  summand, and is the zero map on the first. Hence the image of the
  map $\alpha\inv\iota$, which equals the kernel of~$\epsilon$, is
  $K_{q-1} \Rd{0} \oplus \{0\}$, and there results another exact
  sequence
  \begin{displaymath}
    0 \rTo \{0\} \oplus K_{q} \Rd{0} \rTo^{i^{+}_{*}} K_{q} \Rp
    \rTo^{\nabla} K_{q-1} \vb^{h_{+}}
    \rTo^{\mathrm{pr}_{1}\alpha\inv\iota} K_{q-1} \Rd{0} \rTo 0 \ .
  \end{displaymath}
  By Theorem~\ref{thm:fibre-terms-are-nil} we have an isomorphism
  $\omega \colon K_{q-1} \vb^{h_{+}} \rTo^{\iso} K_{q-1} \cT{-}$.

  \begin{equation}
    \label{eq:claim}
    \begin{gathered}
      \text{\textit{Claim:\/} The composition
        $\mathrm{pr}_{1}\alpha\inv\iota \omega\inv$ is induced \qquad\qquad}
      \\
      \text{by the
        forgetful functor $o^{-} \colon (Y, \upsilon) \mapsto Y$.\quad}
    \end{gathered}
  \end{equation}

  Assuming the claim for the moment, we obtain an exact sequence
  \begin{displaymath}
    K_{q} \Rd{0} \rTo^{i^{+}_{*}} K_{q} \Rp \rTo^{\omega\nabla} K_{q-1}
    \nil-{q-1} \Rd{0} \rTo^{o^{-}} K_{q-1} \Rd{0} \ ,
  \end{displaymath}
  giving the isomorphism
  \begin{displaymath}
    \NK+q = \coker i^{+}_{*} = K_{q} \Rp / \ker(\omega\nabla) \iso \im
    (\omega\nabla) = \ker(o^{-}) \ .
  \end{displaymath}
  This establishes~\eqref{eq:nk+}. The isomorphism of~\eqref{eq:nk-}
  is verified using a symmetric argument, employing the isomorphism
  \begin{equation}
    \begin{multlined}
      \qquad \qquad K_{q} \Rd{0} \oplus K_{q} \Rd{0} \rTo K_{q} \pp %
      \ , \\ %
      \big( [P], [Q] \big) \mapsto \big[ \Psi_{0,0} (P) \big] - \big[
      \Psi_{0,-1} (P) \big] %
      + \big[ \Psi_{0,0} (Q) \big] \ , \qquad \qquad
    \end{multlined}
  \end{equation}
  in place of~$\alpha$.

  \smallbreak

  It remains to verify Claim~\eqref{eq:claim}. By
  Lemma~\ref{eq:functor_A} the isomorphism $\alpha\inv$ is induced by
  the functor
  \begin{displaymath}
    \Xi \colon \vb_{0} \rTo \chp \times \chp \ , \ \cY \mapsto
    \begin{pmatrix}
      \Sigma \Gamma \cY(1,0) \,\oplus\, \Gamma \cY \,\oplus\, \Gamma
      \cY(1,-1) \\
      \Sigma \Gamma \cY(1,-1) \,\oplus\, \Gamma \cY(1,0) \ .
    \end{pmatrix}
  \end{displaymath}
  Now consider the diagram in Fig.~\ref{fig:nil}. It is not a
  commutative diagram of functors, but on the level of $K$\nbd-theory
  spaces it is homotopy commutative. Since all vertical arrows marked
  as inclusion maps induce identity maps on $K$\nbd-groups, and since
  $\omega$ induces an isomorphism on $K$\nbd-groups, this establishes
  Claim~\eqref{eq:claim}. --- %
  \begin{figure}[ht]
    \centering
    \begin{diagram}
      \vb^{h_{+}}_{0} & \rTo[l>=2em]^{\iota} & \vb_{0} &
      \rTo[l>=4em]^{\mathrm{pr}_{1} \Xi} & \chp \\ %
      \dTo>{\subseteq}<{a} && \dTo>{\subseteq}<{b} \\
      \vb^{h_{+}} & \rTo[l>=2em]^{\iota} & \vb && \dTo>{\subseteq}<{i}
      \\%
      \dTo<{\omega} &&&& \\ %
      \cT{-} && \rTo^{o^{-}} && \fdr(\Rd{0}) %
    \end{diagram}
    \caption{Diagram used to establish Claim~\eqref{eq:claim}}
    \label{fig:nil}
  \end{figure}
  \newcommand{\holim}{\mathrm{holim}}%
  The vertical arrows $a$ and~$b$ result in homotopy equivalences on
  $K$\nbd-theory spaces with respect to $h$\nbd-equivalences; this is
  analogous to Lemma~\ref{lem:reduce_to_vb0}), and the arguments of
  Lemma~III.1.3 and Corollary~III.1.4 of~\cite{P1} carry over
  verbatim. For the arrow~$i$ see~\eqref{eq:FD}. As to the alleged
  homotopy commutativity, recall that for $\cZ \in \vb_{0}$ we have
  $\invlim{}\!^{1} \cZ(k,\ell) = H^1 \cZ(k,\ell) = 0$ when
  $k+\ell \geq 0$, by definition of the category ~$\vb_{0}$;
  consequently, the map $\Gamma\cZ(k,\ell) \rTo \holim \cZ(k,\ell)$ is
  a quasi-isomorphism when $k+\ell \geq 0$, where ``$\holim$'' stands
  for the homotopy limit or homotopy pullback construction, which is
  the dual of the double mapping cylinder. Furthermore, for
  $\cZ \in \vb^{h_{+}}_{0}$ we have
  $\cZ(k,\ell)^{+} = Z^{+} \tensor_{\Rp} \Rpn{-\ell} \simeq *$ and
  $\cZ(k,\ell)^{0} = Z^{0} \tensor_{\R} \R \simeq *$, which for
  $k+\ell \geq 0$ implies
  \begin{multline*}
    \Gamma\cZ(k,\ell) \simeq \holim \big( Z^{-} \tensor_{\Rn} \Rnn k
    \rTo Z^{0} \tensor_{\R} \R \lTo Z^{+} \tensor_{\Rp} \Rpn {-\ell}
    \big) \\ %
    \simeq \holim \big( Z^{-} \tensor_{\Rn} \Rnn k \rTo * \lTo * \big)
    \simeq Z^{-} \tensor_{\Rn} \Rnn k \ .
  \end{multline*}
  Thus the composition $i \mathrm{pr}_1 \Xi\,\iota$, that is, the
  functor that sends $\cY \in \vb_{0}^{h_{+}}$ to
  \begin{displaymath}
   \Sigma \Gamma \cY(1,0) \,\oplus\, \Gamma \cY \,\oplus\, \Gamma
      \cY(1,-1) \ ,
  \end{displaymath}
  is weakly equivalent to the functor that sends $\cY$ to
  \begin{displaymath}
    \big( \Sigma Y^{-} \tensor_{\Rn} \Rnn 1 \big) %
    \ \oplus \ Y^{-} \ \oplus \ %
    \big( Y^{-} \tensor_{\Rn} \Rnn 1 \big) \ ;
  \end{displaymath}
  by the additivity theorem, the induced map on $K$\nbd-theory spaces
  is homotopic to the map induced by $\cY \mapsto Y^{-}$ since
  suspension models an inverse for the $H$\nbd-space structure. But
  the formula $\cY \mapsto Y^{-}$ describes the effect
  of~$o^{-} \circ \omega$ as well. This finishes the proof.
\end{proof}

\raggedright


\newcommand{\etalchar}[1]{$^{#1}$}

\end{document}